\documentclass[a4paper,11pt]{amsart}

       \usepackage{palatino, verbatim}
        \usepackage[T1]{fontenc}
        \usepackage[utf8]{inputenc}
        \usepackage{amsthm, amsfonts}
        \usepackage{amsfonts}
        \usepackage{graphicx}
        \usepackage{amssymb}
        \usepackage{amsmath}
        \usepackage{latexsym}
        \usepackage[all]{xy}
        \usepackage{tikz}
        \usepackage{tikz-cd}
        \usepackage{subcaption}
        \usepackage{multirow}
        \usepackage{caption}
\usepackage{subcaption}        % modern subfigure support
  \newtheorem{theorem}{Theorem}[section]
          \newtheorem{corollary}[theorem]{Corollary}
          \newtheorem{lemma}[theorem]{Lemma}
          \newtheorem{proposition}[theorem]{Proposition}
          
           \newtheorem{definition}{Definition}

        \theoremstyle{definition}
         
        \newtheorem{remark}{Remark}
          \newtheorem{example}{Example}[theorem]

\def\P{\mathbb{P}}

\def\Cliff{\mathrm{Cliff}}
\def\deg{\mathrm{deg}}
\def\Pic{\mathrm{Pic}}
\def\dim{\mathrm{dim}}
\def\O{\mathcal{O}}

\title%[The Gieseker-Petri divisor]
{The Prym-canonical Clifford index}
\author{Margherita Lelli--Chiesa, Martina Miseri}
\address{Università Roma Tre, Dipartimento di Matematica e Fisica, 00146 Roma}
\email{margherita.lellichiesa@uniroma3.it\\ martina.miseri@uniroma3.it}

\begin{document}
\begin{abstract}
We introduce two new invariants of Prym curves, the Prym-canonical Clifford index and the Prym-canonical Clifford dimension. The former is a nonnegative integer (according to Prym-Clifford's theorem), while the latter is a pair of nonnegative ordered integers. We classify Prym curves with Prym-canonical Clifford index equal to $0,1,2$. By specialization to hyperelliptic  curves, we compute the Prym-canonical Clifford index of a general Prym curve and show that its Prym-canonical Clifford dimension is $(0,0)$.
\end{abstract}
\maketitle
\section{Introduction}
The Clifford index is a fundamental and widely studied invariant of algebraic curves, which captures key aspects of the geometry of their canonical model. On the one hand, it is determined by the Brill--Noether theory of the curve and measures how far the curve is from being hyperelliptic, thus inducing, by lower semicontinuity, an interesting stratification of the moduli space $\mathcal M_g$. On the other hand, it is expected to govern the syzygies of canonical curves, as predicted by Green's conjecture, first formulated in \cite{green} and proved by Voisin for a general curve \cite{voisineven}, \cite{voisinodd} (cf. \cite{papadima} and \cite{kemeny} for more recent proofs).

In studying Prym-canonical models of algebraic curves, that is, embeddings provided by the tensor product of the canonical line bundle $\omega_C$ with a $2$-torsion line bundle $\eta$, it is natural to look for an analogous invariant. This should be related both to the syzygies of the Prym-canonical curve $\varphi_{\omega_C\otimes \eta}(C)\subset \P(H^0(\omega_C\otimes \eta)^\vee)=\P^{g-2}$ and to the secant varieties
\[
V^{e-f}_e(\omega_C\otimes \eta):= \{ D \in C_e \ | \ h^0(C, \omega_C\otimes\eta(-D))\ge g-1-e+f \},\,\,\,\,\,1\le f<e,\]
which are natural generalizations of the Brill--Noether varieties $W^r_d(C)$ and parametrize degree-$e$ effective divisors imposing $f$ conditions less than expected to the Prym-canonical linear system $|\omega_C\otimes \eta|$ (cf. \cite[Thm 1.1]{farkaslc} for recent results concerning their emptiness/nonemptiness and dimension for a general Prym curve).

If $(C, \eta)$ is a Prym curve and $L \in \Pic(C)$ is a line bundle satisfying $h^0(L)\ge 1$, $h^0(L \otimes \eta)\ge 1$, we set \[\begin{split}
\Cliff_{\eta}(L)&:= h^0(\omega_C \otimes \eta)-h^0(L)-h^0(\omega_C \otimes L^{\vee})+1\\
&=\deg(L)-h^0(L)-h^0(L\otimes \eta)+1.\end{split}
\] We define the {\em Prym-canonical Clifford index of $(C, \eta)$} as the minimum of $\Cliff_{\eta}(L)$ among all line bundles $L$ on $C$ such that $h^i(L)\ge 1$, $h^i(L \otimes \eta)\ge 1$ for $i=0,1$. A line bundle $L$ satisfying these conditions is said to {\em contribute} to the Prym-canonical Clifford index of $(C,\eta)$; if in addition $\Cliff_{\eta}(L)=\Cliff_\eta(C)$, we say that $L$ {\em computes} the Prym-canonical Clifford index. The \textit{Prym-canonical Clifford dimension} of $(C,\eta)$ is then defined as the minimal pair $(r(L), r(L\otimes \eta))$, where $L$ ranges among the line bundles computing the Prym-canonical Clifford index; this invariant is the Prym analogue of the Clifford dimension of an algebraic curve first introduced in \cite{elms}. Since the difference map $\phi_d : C^d \times C^d  \rightarrow \Pic^0(C)$ defined by $\phi_d(D,D'):=\O_C(D-D')$ is surjective for $d =\lfloor \frac{g+1}{2} \rfloor$ (cf. \cite[ch. 5]{acgh}), we obtain the following upper bound for the Prym-canonical Clifford index:
$$
\Cliff_\eta(C)\le  \lfloor \frac{g-1}{2} \rfloor. $$

We prove that the Prym-canonical Clifford index is nonnegative and classify Prym curves with vanishing Prym-canonical Clifford index.
\begin{theorem}(Prym--Clifford's theorem)\label{clifford thm}
    Let $(C,\eta)$ be a Prym curve of genus $g\ge 2$. Then \[ \Cliff_{\eta}(C) \ge 0, \] and equality holds if and only if $\omega_C\otimes \eta$ has some base points, that is, when $C$ is hyperelliptic and $\eta=\O_{C}(p-q)$ with $p$ and $q$ Weierstrass points.
\end{theorem} The characterization of Prym curves $(C,\eta)$ such that $\Cliff_{\eta}(C)= 0$ relies on the following identity: \begin{equation*} 
    2 \Cliff_{\eta}(L)= \Cliff(L)+ \Cliff(L \otimes \eta)-2,
\end{equation*}
which relates the Clifford index and the Prym-canonical Clifford index of any line bundle $L$ contributing both to the Clifford index of $C$ and to the Prym-canonical Clifford index of $(C,\eta)$. This identity also plays a key role in the classification of Prym curves with Clifford index equal to $1$, $2$, as stated in the following two theorems. 
\begin{theorem}\label{thm1}
    Let $(C,\eta)$ be a Prym curve of genus $g\ge 3$. Then, the equality $\Cliff_{\eta}(C)=1$ holds if and only if $\omega_C\otimes \eta$ is base point free but not very ample, that is, when $C$ has a $g^1_4$ and $\eta=\O_C(p+q-x-y)$, with $2(p+q)\sim 2(x+y)$. 
\end{theorem}
\begin{theorem}\label{thm2}
    Let  $(C,\eta)$ be a Prym curve of genus $g\ge 4$. Then, $\Cliff_{\eta}(C)=2$ if and only if the Prym-canonical morphism is an embedding and $\varphi_{\eta}(C)$ has a trisecant line.  
\end{theorem}

The above results show that the Prym-canonical Clifford index of $(C,\eta)$ depends not only on the curve $C$, but also on the choice of the $2$-torsion line bundle $\eta$. Furthermore, they highlight the close connection between this invariant and the geometry of the Prym-canonical map $\varphi_{\omega_C\otimes\eta}$. This relation becomes particularly transparent in the case of hyperelliptic Prym curves. By \cite[Lem. 4.3]{verra}, every non-trivial $2$-torsion line bundle $\eta$ on a hyperelliptic curve $C$ can be written as\begin{equation*}
\eta= \O_C(w_1+\ldots+ w_k- w_{k+1}- \cdots - w_{2k}),
\end{equation*} where $1 \le k \le \lfloor \frac{g+1}{2} \rfloor$ and $w_1, \ldots, w_{2k}$ are distinct Weierstrass points. This expression is unique up to the reordering of the points $w_i$, except in the case where $k$ is maximal, in which two complementary choices of the set $\{w_1, \ldots, w_{2k}\}$ define the same line bundle $\eta$.  It turns out that the Prym-canonical Clifford index of $(C,\eta)$ depends only on the integer $k$:
\begin{theorem}\label{contoiper}  If $C$ is hyperelliptic and $\eta= \O_C(w_1+\ldots+ w_k- w_{k+1}- \ldots - w_{2k})$, with $w_1, \ldots, w_{2k}$ distinct Weierstrass points of $C$ and $1\le k \le \lfloor \frac{g+1}{2} \rfloor$, then: \[\Cliff_{\eta}(C)=k-1\,\,\,\,\,\textrm{       and      }\,\,\,\,\,\, \dim \Cliff_{\eta}(C)=(0,0).\]
\end{theorem}
In particular, by varying the line bundle $\eta$ on a fixed hyperelliptic curve $C$, one obtains all possible values for $\Cliff_\eta(C)$. Since the Prym-canonical Clifford index is lower semicontinous, it follows that the upper bound $\lfloor \frac{g-1}{2} \rfloor$ is attained by a general element in the moduli space $\mathcal R_g$ of Prym curves of genus $g$:
\begin{theorem} \label{Cgen}
    Let $(C, \eta) \in \mathcal{R}_g$ be a general Prym curve of genus $g$. Then: \[
\Cliff_{\eta}(C)= \lfloor \frac{g-1}{2} \rfloor ,\,\,\,\,\textrm{       and      }\,\,\,\,\,\, \dim \Cliff_{\eta}(C)=(0,0).
\]
\end{theorem}
The hyperelliptic case also sheds light on the relationship between the Prym-canonical Clifford index and the syzygies of a Prym-canonical curve. Indeed, one easily shows that the integer $k$ (or equivalently, the value of $\Cliff_\eta(C)$) uniquely determines the rational normal scroll $S$ containing the Prym-canonical image of $C$ and hence, by Theorem \cite[Thm 1.3]{parkhyper}, its syzygies. The fact that the syzygies of a Prym-canonical curve are influenced by its secant divisors (and thus by its Prym-canonical Clifford index) is already reflected in known results. In particular, Green–Lazarsfeld’s theorem on the normal generation of Prym-canonical curves of Clifford index $2$ \cite[Thm 2.1]{greenlazarsfeld} and Lange–Sernesi’s result on the quadratic generation of the ideal of a Prym-canonical curve of Clifford index $4$ \cite[Lem. 2.1]{langesernesi} both point in this direction. Moreover, the same works (more precisely, \cite[Thm 1]{greenlazarsfeld} and \cite[Cor. 2.3]{langesernesi}) show that the syzygies of a Prym-canonical curve also reflect the classical Clifford index. These observations motivate the introduction of a further invariant. Denoting by $\pi:\tilde C\to C$ the étale double cover associated with a Prym curve $(C,\eta)$ and by $\iota$ the induced involution on $\tilde C$, we define the {\em $\iota$-invariant Clifford index of $\tilde C$} as the minimum of $\Cliff(\tilde L)$, where $\tilde L$ ranges over the $\iota$-invariant line bundles on $\tilde C$ contributing to the Clifford index of $\tilde C$. We prove (cf. Proposition \ref{quasi}) that this invariant takes track of both the Clifford index of $C$ and the Prym-canonical Clifford index of $(C,\eta)$. We believe that the relationship between the $\iota$-invariant Clifford index and syzygies deserves further investigation.

We conclude with a remark on the Prym-canonical Clifford dimension. Our classification results show that every Prym curve with Prym-canonical Clifford index $\le 2$ has minimal Prym-canonical Clifford dimension, namely, $(0,0)$. This stands in sharp contrast with the classical situation, where there exist curves of Clifford index $1$ and non-minimal Clifford dimension (so-called exceptional according to \cite{elms}), namely, smooth plane quintics. A natural open question is therefore the existence of Prym curves with Prym-canonical Clifford dimension strictly greater than $(0,0)$, which one may natural call \textit{Prym-exceptional}; we plan to address this problem in future work.

\subsection{Structure of the paper.} Section \ref{preliminary} recalls preliminary results on the Prym-canonical linear system, as well as bounds due to Coppens-Martens on the degree of line bundles computing the classical Clifford index of an algebraic curve. 

In Section \ref{tre} we introduce the Prym-canonical Clifford index and dimension, and we establish several useful properties. In particular, we show that any line bundle computing the Prym-canonical Clifford dimension of $(C,\eta)$ is base point free (cf. Lemma \ref{lemma3}), and that, whenever $\dim \Cliff_{\eta}(C) \neq (0,0)$, one has $\dim \Cliff_{\eta}(C) \ge (1,2)$ (cf. Propositions \ref{lemma1} and \ref{lemma4}). We also describe the relationship between the Prym-canonical Clifford index and the secant varieties $V^{e-f}_e(\omega_C \otimes \eta)$, and we prove Prym--Clifford's theorem. In the final paragraph of the section, we consider the étale double cover $\pi \colon \tilde{C} \to C$  induced by $\eta$ and the covering involution $\iota$  on $\tilde{C}$. We then introduce the $\iota$-invariant Clifford index of $\tilde C$. Proposition \ref{quasi} computes this invariant in terms of the Prym-canonical Clifford index of $(C,\eta)$ and the gonality of $C$.

Section \ref{quattro} begins with some background on linear series on curves of low gonality. We then prove that, if $C$ is bielliptic, the Prym-canonical Clifford dimension of $(C,\eta)$ is always $(0,0)$. The remainder of the section is devoted to the proof of the classification results for Prym curves with Prym-canonical Clifford index $1$ and $2$. These proofs are quite involved and rely on the description of linear series on  curves of gonality $3$, $4$ and $5$, as well as on the geometry of plane curves with double points.

In Section \ref{cinque} we focus on hyperelliptic curves, proving Theorem \ref{contoiper} and relating it to the syzygies of hyperelliptic Prym-canonical curves. In analogy with the classical Clifford index, we then show that the Prym-canonical Clifford index is lower semicontinuous in families of Prym curves, and we finally prove Theorem \ref{Cgen} for a general Prym curve.

\vskip 5pt

\noindent {\small{{\bf Acknowledgments:} We are grateful to Andrea Bruno, Andreas Leopold Knutsen and Alessandro Verra for numerous valuable conversations related to the topics of this paper. The authors were partially supported by PRIN2022 “Moduli spaces and Birational Geometry” and by the INdAM research group GNSAGA.}}

\section{Preliminaries}\label{preliminary}
Let $(C,\eta)$ be a Prym curve of genus $g \ge 2$, that is, $C$ is a smooth genus $g$ curve over $\mathbb{C}$ and $\eta \in \mathrm{Pic}^0(C)$  is a nontrivial $2$-torsion line bundle. We denote by \[\varphi_{\omega_C\otimes\eta} \colon C \rightarrow \mathbb{P}^{g-2} \] the morphism defined by the \textit{Prym-canonical system} $|\omega_C \otimes \eta|$. 
We recall some preliminary results about the Prym-canonical system.
\begin{lemma}(\cite[Lem. 2.1]{cdgk}) \label{first lemma}
Let $(C,\eta)$ be a Prym curve of genus $g\geq 3$. Then: \begin{enumerate}
\item[(i)] $p$ is a base point of $|\omega_C \otimes \eta|$ if and only if $|p+\eta|\ne \emptyset$. This happens if and only if $C$ is hyperelliptic and $\eta \sim \O_C(p-q)$, with $p$ and $q$ ramification points of the $g^1_2$. In particular, $p$ and $q$ are the only base points; \label{first condition}
\item[(ii)] if $|\omega_C \otimes \eta|$ is base point free, then it does not separate $p$ and $q$ (possibly infinitely near) if and only if $|p+q+ \eta| \ne \emptyset$. This happens if and only if $C$ has a $g^1_4$ and $\eta \sim \O_C(p+q-x-y)$, where $2(x+y)$ and $2(p+q)$ are members of the $g^1_4$. In particular, also $x$ and $y$ are not separated by $|\omega_C \otimes \eta|$. \label{second condition}
\end{enumerate}

\end{lemma}

\begin{corollary}(\cite[Cor. 2.2]{cdgk}) \label{coro}
Suppose the Prym-canonical system is base-point free and $C$ has genus $g \ge 4$. If $\varphi_{\omega_C\otimes\eta} \colon C \rightarrow  \P^{g-2}$ is not birational onto its image, then it is of degree $2$ onto a smooth elliptic curve $E \subseteq \P^{g-2}$ (hence, $C$ is bielliptic). Moreover, $\eta$ is the pullback of a nontrivial $2$-torsion line bundle on $E$. 
\end{corollary}

\begin{remark}
    When $g=2$, the Prym-canonical system $|\omega_C \otimes \eta|$ cannot be base-point free, as $\varphi_{\omega_C\otimes\eta}$ sends $C$ to a point. Analogously, if $g=3$ and $\omega_C\otimes \eta$ is base point free, then the morphism $\varphi_{\omega_C\otimes\eta}$ is never birational onto its image since its image is $\P^1$.
\end{remark}

\begin{remark} \label{notbir}
By the previous corollary, if $\varphi_{\omega_C\otimes\eta}$ is not birational, then it factors as
\[ \begin{tikzcd}
C \arrow[r, "f"] \arrow[d, "\varphi_{\omega_C\otimes\eta}"'] & E \arrow[ld] \\
\mathbb{P}^{g-2}                              &             
\end{tikzcd}
\]
where $f$ is a degree $2$ cover on a smooth elliptic curve $E$. In particular, every divisor in $|\omega_C \otimes \eta|$ is the pullback of a hyperplane section of $E \subseteq \P^{g-2}$. \end{remark}

Throughout the paper we will also make use of the following results by Coppens-Martens bounding the degree of line bundles computing the Clifford index of a smooth curve.
\begin{theorem}\cite[Thm C]{cm} \label{thm c}
    Let $C$ be a curve of genus $g$. If there is a $g^r_d$  computing the Clifford index of $C$ and $d \le g-1$, then $d \le 2(\Cliff(C)+2)$ unless $C$ is hyperelliptic or bielliptic.
\end{theorem}
 
 \begin{corollary}\cite[Cor. 3.2.5]{cm} \label{cor thm c}
    Let $C$ be a curve of genus $g>2\Cliff(C)+4$ (respectively, \mbox{$g>2\Cliff(C)+5$}) if $\Cliff(C)$ is odd (resp., even). If there is a $g^r_d$ with $d \le g-1$ computing the Clifford index of $C$, then $d \le 3(\Cliff(C)+2)/2$ unless $C$ is hyperelliptic or bielliptic.
 \end{corollary}
 \begin{remark}\label{cliffbassi}
  Let $L$ be a line bundle of degree $d\le g-1$ computing the Clifford index of $C$. The above results yield that, if $\Cliff(C)= 1$, then $L$ is either a $g^1_3$ or a $g^2_5$; since the arithmetic genus of a plane quintic is $6$,  our requirement on $d$ implies $g=6$ in the latter case (that is, the $g^2_5$ is automatically very ample). 
  
  Analogously, if $\Cliff(C)= 2$, then $L$ is of type either $g^1_4$, or $g^2_6$, or $g^3_8$, and the latter case may only occur for $g=9$. If we further require that $L$ computes the Clifford dimension of $C$, we can exclude that $L$ is a $g^3_8$ because curves of Clifford dimension $3$ are complete intersections of two cubics in $\P^3$ and have genus $10$ (cf. \cite{elms}). We recover the well-known fact that, if $\Cliff(C)= 2$, then $C$ is either tetragonal or a smooth plane sextic. 
  
  Finally, if $\Cliff(C)= 3$, then any line bundle $L$ of degree $d\le g-1$ computing the Clifford index is forced to be either a $g^1_5$, or a $g^2_7$, or, only when $g=10$, a $g^3_9$. If we further require that $L$ computes the Clifford dimension of $C$, we recover that, if $\Cliff(C)=3$, then  either $C$ is $5$-gonal, or $g=15$ and $C$ is a smooth plane septic, or $g=10$ and $C$ is a complete intersection of two cubics in $\mathbb P^3$.
  \end{remark}

\section{Definitions and first results}\label{tre}
\subsection{The Prym-canonical Clifford index and dimension} 

We introduce some new invariants associated with any Prym-curve.  
\begin{definition}
Let $(C,\eta)$ be a Prym curve of genus $g\geq 2$. Given $L \in \mathrm{Pic}(C)$ such that $h^0(L) \ge 1$, $h^0(L \otimes \eta) \ge 1$, the \emph{Prym-canonical Clifford index of $L$ with respect to $\eta$} is \[\Cliff_{\eta}(L):= h^0(\omega_C \otimes \eta) - h^0(L)-h^0(\omega_C \otimes \eta \otimes L^{\vee}) +1. \] 
\end{definition}
Denoting by $r(L)$ the dimension of the linear system $|L|$, the previous definition can be rewritten as \[\mathrm{Cliff}_{\eta}(L)= r(\omega_C \otimes \eta) - r(L)-r(\omega_C \otimes \eta \otimes L^{\vee}) . \]  
 \begin{definition}\label{oki}
The \emph{Prym-canonical Clifford index of $(C,\eta)$} is 
\begin{align}\label{cliffeta}
\Cliff_{\eta}(C):&= \mathrm{min} \{\mathrm{Cliff}_{\eta}(L) \ |   h^i(L) \ge 1, h^i(L \otimes \eta) \ge 1,\,\,i=0,1  \}=\\\nonumber&=\mathrm{min} \{\mathrm{Cliff}_{\eta}(L) \ |  \mathrm{deg}(L) \le g-1, h^0(L) \ge 1, h^0(L \otimes \eta) \ge 1  \}. 
\end{align}
A line bundle $L$ is said to {\em contribute to the Prym-canonical Clifford index} if both $h^i(L) \ge 1$ and $h^i(L \otimes \eta) \ge 1$ for $i=0,1$, and to {\em compute the Prym-canonical Clifford index} if  moreover $\Cliff_{\eta}(C)=\Cliff_{\eta}(L)$.
\end{definition}

\begin{remark}  The second equality in \eqref{cliffeta} follows from the fact that $\mathrm{Cliff}_{\eta}(L)=\mathrm{Cliff}_{\eta}(\omega_C \otimes \eta \otimes L^{\vee})$ and that any effective line bundle of degree $\leq g-1$ is special. Requiring $\mathrm{deg}(L) \le g-1$ as in the second line of \eqref{cliffeta} thus avoids considering both $L$ and $\omega_C \otimes \eta \otimes L^{\vee}$.\\Also note that, by Riemann-Roch, we obtain \[ \begin{split}
\Cliff_{\eta}(L) & = \deg(L)-h^0(L)-h^0(L\otimes \eta)+1 \\
& = \Cliff_{\eta}(L \otimes \eta) \\
& =\Cliff_{\eta}(\omega_C \otimes L^{\vee}).
\end{split}
\] 

\end{remark}
 
 \begin{remark}In the classical sense, a line bundle $L$ of degree $d\leq g-1$ is said to contribute to the Clifford index of $C$ if $h^0(L)\ge 2$, or equivalently, there exist two effective divisors $D, D' \in |L|$ such that the trivial bundle $\O_C \simeq \O_C(D-D')$. Our conditions $h^0(L) \ge 1$ and $h^0(L \otimes \eta)\ge 1$ in Definition \ref{oki} are a natural generalization as they are equivalent to the existence of two effective divisors $D\in |L|$ and $D'\in |L\otimes \eta|$ such that $\eta \simeq \O_C(D-D')$.\end{remark}

\begin{remark} \label{surj_diff_map}
    Since the difference map \[  \begin{split}
    \phi_d \colon C^d \times C^d & \rightarrow \Pic^0(C) \\
    (D, D') & \mapsto \O_C(D-D')
    \end{split}
    \] is surjective for all $d \ge \lfloor \frac{g+1}{2} \rfloor$ (cf. \cite[ch. 5]{acgh}), it is always possible to write $\eta= \O_C(D-D')$ with $D, D'$ effective of degree $\lfloor \frac{g+1}{2} \rfloor$. It follows that \begin{equation} \label{dis_diff_map}
    \Cliff_{\eta}(C)  \le \lfloor \frac{g+1}{2} \rfloor-1  
    \end{equation} 
\end{remark}

\begin{example} \label{hyperelliptic example}
Let $C$ be hyperelliptic and $\eta= \O_C(p-q)$ with $p, q$ ramification points of the $g^1_2$. Setting $L:=\O_C(p)$, we have $L \otimes \eta\simeq\O_C(2p-q)\simeq\O_C(q)$ so that $h^0(L)=h^0(L \otimes \eta)=1$ and $\Cliff_{\eta}(L)=0$. 
\end{example}

\begin{example} \label{ex2}
We now consider the case where $\omega_C \otimes \eta$ is base point free and $\eta= \O_C(p+q-x-y)$ with $2(p+q)\sim 2(x+y)$. Fix $L:=\O_C(p+q)$ so that $L \otimes \eta\simeq\O_C(2p+2q-x-y)\simeq\O_C(x+y)$. We claim that $ h^0(L)=1$. Indeed, if $L$ were a $g^1_2$, denoting by $\iota$ the hyperelliptic involution, we could write $\eta\simeq \O_C(\iota(x)-y)$; since $\eta^{\otimes 2}\simeq \mathcal O_C$, then $\iota(x)=x$ and $y$ would be ramifications points and this is a contradiction because, by Lemma \ref{first lemma}, $\omega_C\otimes \eta$ would have base points. Analogously, one shows that $h^0(L\otimes \eta)=1$ and thus $\Cliff_{\eta}(L)=1$. Theorem \ref{thm1} will then imply $\Cliff_{\eta}(C)=1$.
\end{example}

\begin{example} \label{biellipticex}
    Let $C$ be a bielliptic curve (that is, it admits a degree $2$ morphism $f \colon C \to E$ to an elliptic curve $E$) and consider a $2$ torsion line bundle on $C$ of the form $\eta= f^*(\epsilon)$ for some $\epsilon \in \Pic^0(E)[2] \setminus \{\O_E\}$. Since the Abel-Jacobi map $E \rightarrow \Pic^0(E)$ is an isomorphism, we can write $\epsilon\simeq\O_E(x-y)$ for some $x,y \in E$ such that $2x \sim 2y$, and $\eta\simeq\O_C(x_1+x_2-y_1-y_2)$, where $x_1,x_2$ (respectively, $y_1,y_2$) are the inverse images under $f$ of $x$ (resp., $y$) so that $2(x_1+x_2)\sim 2(y_1+y_2)$. We have thus fallen into the case covered by Example \ref{ex2} and $\Cliff_{\eta}(C)=1$.

\end{example}
We now introduce the following:
\begin{definition}
The \emph{Prym-canonical Clifford dimension of $C$} is defined as   
$$
\dim \Cliff_{\eta}(C):=\mathrm{min} \{(r,r') \ | \  \exists L\in \Pic(C),\,\ \deg(L)\le g-1, \  r=r(L), r'=r(L \otimes \eta), \Cliff_{\eta}(C)= \Cliff_{\eta}(L) \},
$$
where the minimum is taken with respect to the lexicographic order on $\mathbb Z^2$. 

We say that $L \in \Pic(C)$ {\em computes the Prym-canonical Clifford dimension} if $\Cliff_{\eta}(C)= \Cliff_{\eta}(L) $ and $\dim \Cliff_{\eta}(C)= (r(L),r(L \otimes \eta))$.
\end{definition}
\begin{remark}\label{compute}
Since the order is lexicographic and $\Cliff_{\eta}(L)=\Cliff_\eta(L\otimes \eta)$ for any $L$, it easily follows that a line bundle $L$ computing the Prym-canonical Clifford dimension satisfies  \mbox{$0\leq r(L)\leq r(L\otimes \eta)$}.
\end{remark}

\begin{proposition}\label{lemma1}
    For every Prym curve $(C,\eta)$ and any $r'\geq 1$, one has \[
    \dim \Cliff_{\eta}(C) \ne (0,r').
    \]
\end{proposition} 

\begin{proof}
By contradiction, assume $  \dim \Cliff_{\eta}(C) = (0,r')$ for some $r'\geq 1$. Let $L$ be a line bundle computing the Prym-canonical Clifford dimension of $C$ and take $D=q_1+ \ldots +q_d \in |L|$.  Since $h^0(L)=1$, all the $q_i$ are base points of $L$ and, setting $L_1=L(-q_1)$, we have $h^0(L_1)=1$ and $h^0(L_1\otimes \eta)=h^0(L\otimes \eta(-q_1))\geq r'-1\geq 0$. We get
\[\begin{split}
\Cliff_{\eta}(L_1)& = \deg(L_1)-h^0(L_1)-h^0(L_1\otimes\eta)+1 \\
& \le \deg(L)-1-h^0(L)-h^0(L\otimes \eta)+1+1\\
& = \Cliff_{\eta}(L)=\Cliff_{\eta}(C),
\end{split}
\]
and thus equality holds. In conclusion, we have $\Cliff_{\eta}(L_1)=\Cliff_{\eta}(C)$ and $(r(L_1),r(L_1\otimes\eta))=(r(L),r(L\otimes\eta)-1)<(r(L),r(L\otimes\eta))$ in contradiction with the assumption that $L$ computes the Prym-canonical Clifford dimension of $C$.
\end{proof}

\begin{lemma}
\label{lemma3}
Let $L$ be a line bundle computing the Prym-canonical Clifford index of $(C,\eta)$. Then,  $L$ and $L\otimes \eta$ have no base points in common. \\If moreover $\dim \Cliff_{\eta}(C) \ge (1,1)$ and $L$ computes the Prym-canonical Clifford dimension of $(C, \eta)$, then both $L$ and $L\otimes \eta$ are base point free.
\end{lemma}
\begin{proof}
By  contradiction, assume $p$ is a base point of both $L$ and $L\otimes \eta$ and set $L_1:=L(-p)$. Since $h^0(L_1)=h^0(L)$ and $h^0(L_1\otimes\eta)= h^0(L\otimes \eta)$, the line bundle $L_1$ contributes to the Prym-canonical Clifford index and one easily computes that $\Cliff_\eta (L_1)=\Cliff_\eta(L)-1$, in contradiction with the equality $\Cliff_\eta(C)=\Cliff_\eta(L)$. \\ Assume $\dim \Cliff_{\eta}(C)=(r(L), r(L \otimes \eta)) \ge (1,1)$ and $p$ a base point of $L$. Setting $L_1=L(-p)$, $h^0(L_1)=h^0(L)$ and $h^0(L_1 \otimes \eta)=h^0(L(-p)\otimes \eta) \ge h^0(L \otimes \eta)-1 \ge 1$. We get \[
\begin{split}
\Cliff_{\eta}(L_1)&= \deg(L_1)-h^0(L_1)-h^0(L_1 \otimes \eta) +1 \\
& \le \deg(L)- 1 -h^0(L)-h^0(L \otimes \eta) +1 +1 \\
& = \Cliff_{\eta}(L)= \Cliff_{\eta}(C),
\end{split}
\] and thus equality holds. In conclusion, the line bundle $L_1$ computes the Prym-canonical Clifford index and $(r(L_1), r(L_1\otimes \eta))=(r(L), r(L\otimes \eta)-1)<(r(L),r(L\otimes \eta))$ in contradiction with the assumption that $L$ computes the Prym-canonical Clifford dimension. Analogously, one shows that $L \otimes \eta$ is base point free.
\end{proof}

\begin{lemma}
\label{lemma2}
Let $L$ be a line bundle contributing to the Prym-canonical Clifford index of $(C,\eta)$ such that $\deg L\leq g-1$, $h^0(L)\ge 2$ and $h^0(L\otimes \eta)\ge 2$. Then, the following equalities hold:\[ \begin{split}
2 \Cliff_{\eta}(L)&=\Cliff_{\eta}(L) + \Cliff_{\eta}(L\otimes \eta)=   \\
 & = \Cliff(L)+\Cliff(L\otimes \eta)-2.
\end{split}
\] In particular, one gets \begin{equation} \label{fiore}
\Cliff_{\eta}(L) \ge \Cliff(C)-1.
\end{equation}
\end{lemma}

\begin{proof}
Both $L$ and $L\otimes \eta$ contribute to the Clifford index of $C$. The inequality follows directly from the definition.
\end{proof}

To state the next result we recall that, having fixed a line bundle $L \in \Pic^d(C)$ with $h^0(C,L)=r+1$ and positive integers $0 \le f <e$, one introduces the secant variety \[
V^{e-f}_e(L):= \{ Z \in C_e \ : \ h^0(C, L(-Z))\ge r+1-e+f \}
\] parametrizing effective divisors $Z$ of degree $e$ which impose $f$ less than expected to $|L|$. They are a generalization of the Brill--Noether variety $W^r_d(C)$.

\begin{lemma} \label{lemma6}
Let $L$ be a line bundle computing the Prym-canonical Clifford dimension of $(C,\eta)$. Then, for every positive integers $1 \le f < e \le \min\{\deg L, r(L)+f\}$ and  $1 \le f'<e\le \min\{\deg L, r(L\otimes \eta)+f'\}$ such that $f+f' \ge e$, we have\[
V_e^{e-f}(L) \cap V_e^{e-f'}(L \otimes \eta)= \emptyset.
\]
\end{lemma}

\begin{proof}
By contradiction, assume there exists an effective divisor $D \in V_e^{e-f}(L) \cap V_e^{e-f'}(L \otimes \eta)$. Then, \begin{equation} \label{cccp}
\begin{split}
 h^0(L(-D))& \ge h^0(L)-e+f  \ge 1 \\
 h^0(L\otimes \eta(-D)) & \ge h^0(L \otimes \eta)-e+f'\ge 1,
\end{split}
\end{equation}
 so $L(-D)$ contributes to the Prym-canonical Clifford index. Since $L$ computes the Prym-canonical Clifford dimension of $C$, we have $\Cliff_\eta(L)< \Cliff_{\eta}(L(-D))$, or equivalently, \[ 
h^0(L)+ h^0(L \otimes \eta) > e+h^0(L(-D))+h^0(L \otimes \eta (-D)),
\] and inequalities \eqref{cccp} lead to the contradiction \[ h^0(L)+h^0(L \otimes \eta)> h^0(L)+h^0(L \otimes \eta)+ (f+f'-e).\] \end{proof}
The above result can be used to prove the following:
\begin{proposition} \label{lemma4}
 For every Prym curve $(C,\eta)$ of genus $g\geq 2$, one has \[
    \dim \Cliff_{\eta}(C) \ne (1,1).
    \]
\end{proposition}

\begin{proof}
 Let $L$ be a line bundle computing the Prym-canonical Clifford dimension of $C$ and, by contradiction, assume $  \dim \Cliff_{\eta}(C) =(r(L),r(L\otimes \eta))= (1,1)$. Both $L$ and $L\otimes \eta$ are base point free by Lemma \ref{lemma3} and thus define a morphism 
 \[
f \colon C \xrightarrow{|L|, |L \otimes \eta|}\P^1 \times \P^1.
\]
Lemma \ref{lemma6} then forces $f$ to be an embedding. Indeed, if this were not the case, we would find a degree $2$ divisor $D$ such that $h^0(L(-D))=h^0(L\otimes \eta(-D))=1$, that is, $D\in V^1_2(L)\cap V^1_2(L\otimes \eta)$; this contradicts Lemma \ref{lemma6} for $e=2$ and $f=f'=1$.

We have thus reduced to exclude the case where $f$ embeds 
 $C\subset \P^1 \times \P^1$ as a curve of bidegree $(d,d)$ with $d=\deg L$. The cohomology of the following short exact sequence 
\[
0 \rightarrow\O_{\P^1\times \P^1}(2-d, -2-d)\rightarrow \O_{\P^1 \times \P^1}(2, -2)\rightarrow \O_C \rightarrow 0,
\] along with the K\"unneth Formula yields $d=2$, in contradiction with the assumption $g\geq 2$. 
\end{proof}

Lemma \ref{lemma6} has also the following corollary.
%\begin{remark}
%In the Examples \ref{hyperelliptic example}, \ref{ex2}, \ref{biellipticex} the Prym-canonical Clifford dimension of $C$ is $(0,0)$. \\Analogously, Corollary \ref{Cgen} yields that when $C$ is general, then $\dim \Cliff_{\eta}(C)=(0,0)$.
%\end{remark}
%\begin{remark}
%Lemma \ref{lemma6} implies Proposition \ref{lemma4}. Indeed, let $L, L \otimes \eta$ be as in the hypothesis of the lemma above. If $\dim \Cliff_{\eta}(C)= (1,1)$, the integers $e,f$ are such that $e-f=1$. A divisor $D=p_1+ \ldots+p_e \in V_e^1(L)$ is such that $h^0(L-D) \ge h^0(L)-1= 1 $, that is, the points $p_1, \ldots p_e$ are identified by the morphism $\phi_L$ defined by $|L|$. Similarly, the variety $V_e^{e-1}(L \otimes \eta)$ is the set of degree $e$ divisors such that their points are identified by the morphism $\phi_{L \otimes \eta}$ defined by $|L \otimes \eta|$. Hence, there is not a degree $e$ divisor in $C$ whose points are identified by both $\phi_L$ and $\phi_{L \otimes \eta}$. In particular, when $e=2$ this means that given divisors $M \in |L|$ and $M' \in |L \otimes \eta|$, then $M$ and $M'$ do not have a couple of points in common. Thus, if the morphism \[
%C \xrightarrow{|L|, |L\otimes \eta|} \P^1 \times \P^1
%\] is birational, it is an embedding.
%\end{remark}

\begin{corollary}
Assume that $\dim \Cliff_{\eta}(C)=(1,2)$ and let $L$  be a line bundle computing the Prym-canonical Clifford dimension of $(C,\eta)$. Then, no fiber of $\phi_L$ contains a degree $3$ divisor lying in the intersection of $\phi_{L\otimes \eta}(C)\subset  \P^2$ with a line. Furthermore, $\phi_{L\otimes \eta}$ does not identify any pair of points in a fiber of $\phi_L$.
\end{corollary}

\begin{proof}
Since $r(L\otimes \eta)=2$, we may apply Lemma \ref{lemma6} for $1\le e-f'\le2$ and $e-f=1$. Consider the case $e-f'=2$. The hypothesis of Lemma \ref{lemma6} are satisfied when $f+f' \ge e$, that is, $e \ge 3$, in which case we obtain that $V^1_e(L) \cap V^2_e(L \otimes \eta)= \emptyset$, and the statement follows taking $e=3$.

We now consider the case where $e-f'=e-f=1$, so that $f+f' \ge e$ as soon as $e\ge 2$. Lemma \ref{lemma6} then yields $V^1_e(L) \cap V^1_e(L \otimes \eta)= \emptyset$ and the last part of the statement follows choosing $e=2$.
\end{proof}

%\begin{remark}
%Assume that $\dim \Cliff_{\eta}(C)=(1,2)$ and let $L$ be the line bundle computing it. We consider the secant variety $V_{e'}^{e'-f'}(L \otimes \eta)$ with $e'-f'=1$. A divisor $D=p_1+ \ldots p_{e'} \in V_{e'}^1(L \otimes \eta)$ is such that the morphism $\phi_{L \otimes \eta}: C \rightarrow \bar{C}\subset \P^2$ sends the points $p_1, \ldots, p_e'$ in a singularity of $\bar{C}$, as $h^0(L \otimes \eta -D)\ge h^0(L \otimes \eta)-1 \ge 2$. By Lemma \ref{lemma6}, $D \notin V_{e'}^{e'-1}(L)$, that is, $p_1, \ldots p_{e'}$ impose at least two independent conditions on the image of $C$ under the morphism $\phi_L \colon C \rightarrow \P^1$.
%\end{remark}
The Prym-canonical Clifford index can be related to the secant varieties of $\omega_C\otimes \eta$ as follows:

\begin{remark}
    For any fixed integers $1 \le f <e$, we consider the secant variety \[ V^{e-f}_e(\omega_C \otimes \eta)=\{D \in C_e : h^0(C, \omega_C \otimes \eta(-D)) \ge g-1-e+f \}.\]
    One has the following inclusions: \[
V_e^{e-1}(\omega_C \otimes \eta) \supseteq V_e^{e-2}(\omega_C \otimes \eta) \supseteq \ldots \supseteq V_e^{1}(\omega_C \otimes \eta).
\]
    By the Riemann-Roch formula, the condition $D \in V^{e-f}_e(\omega_C \otimes \eta)$ is equivalent to the inequality $h^0(\O_C(D)\otimes \eta)\ge f$. Therefore, an effective line bundle $L=\O_C(D)$ of degree $e\leq g-1$ contributes to the Prym-canonical Clifford index of $(C,\eta)$ if and only if $D \in V_e^{e-1}(\omega_C \otimes \eta)$. If we take a divisor $D \in V_e^{e-f}(\omega_C \otimes \eta)$, then \begin{equation} \label{sss}
\begin{split}
\Cliff_{\eta}(\O_C(D))& = e-h^0(\O_C(D))-h^0(\O_C(D)\otimes \eta)+1\le \\
& \le e-h^0(\O_C(D))-f+1\le e-f,
\end{split}
\end{equation}
where the former (respectively, the latter) inequality is an equality as soon as $D \not\in V_e^{e-f-1}(\omega_C \otimes \eta)$ (resp., $D$ does not move). In particular, if the Prym-canonical Clifford dimension is $(0,0)$, then $\Cliff_{\eta}(C)=e_0-1$ where $e_0$ is the minimal $e$ such that $V_e^{e-1}(\omega_C \otimes \eta)$ is nonempty. We refer to \cite[Thm 1.1]{farkaslc} for nonemptiness results of $V^{e-f}_e(\omega_C \otimes \eta)$ when $(C,\eta)$ is a general genus $g$ Prym curve.
\end{remark}

\subsection{Prym-Clifford's theorem}
We are now ready to prove the analogue of Clifford's theorem.

\begin{proof}[Proof of Theorem \ref{clifford thm}]
Consider $L \in \Pic(C)$ such that $\Cliff_{\eta}(L)=\Cliff_{\eta}(C)$. Since $r(L) \ge 0$ and $r(\omega_C \otimes \eta \otimes L^{\vee}) \ge 0$, the inequality follows from \[ r(L)+r(\omega_C \otimes \eta \otimes L^{\vee}) \le r(\omega_C \otimes \eta) \]  (see \cite[Lemma 5.5, ch. IV]{hartshorne}). Now we prove the second part of the statement.

 When $C$ is hyperelliptic and $\eta$ is as in the statement, the equality $\Cliff_{\eta}(C) = 0$ follows from Example \ref{hyperelliptic example}. 

Viceversa, let $\Cliff_{\eta}(C) = 0$ and let $L$ be a line bundle computing the Prym-canonical Clifford dimension of $(C,\eta)$. If $r(L)=0$, then we have \[r(\omega_C \otimes \eta \otimes L^{\vee})=r(\omega_C \otimes \eta), \] that is, the only effective divisor in $|L|$ lies in the base locus of $|\omega_C \otimes \eta|$. Hence, in this case the conclusion follows from Lemma \ref{first lemma}. We now suppose $r(L)\geq 1$; since $L$ computes the Prym-canonical Clifford dimension of $(C,\eta)$, $r(L\otimes \eta) \ge r(L)\ge 1$ (cf. Remark \ref{oki}). 
Lemma \ref{lemma2} then yields \[
\Cliff(L)+ \Cliff(L \otimes \eta)=2.
\] Since $r(L \otimes \eta)\ge r(L)$, then $\Cliff(L)\ge \Cliff(L\otimes\eta)$; we obtain that $0\le\Cliff(L\otimes\eta)\le 1$. If $\Cliff(L\otimes \eta)=0$ and $\Cliff(L)=2$, then $C$ is hyperelliptic, $L\otimes \eta$ is a multiple of the $g^1_2$ and $L$ is a multiple of the $g^1_2$ plus $2$ base points in contradiction with Proposition \ref{lemma1}. 

 If $\Cliff(L)=\Cliff(L\otimes \eta)=1$, then $C$ cannot be hyperelliptic because otherwise $L$ and $L \otimes \eta$ would have base points. Hence, $\Cliff(C)=1$ and $L$ and $L\otimes \eta$ are either both of type $g^1_3$ or both very ample $g^2_5$ (cf. Remark \ref{cliffbassi}). The former case contradicts Proposition \ref{lemma4}. The latter case can be excluded thanks to the unicity of a $g^2_d$ on a smooth plane curve of degree $d \ge 4$ (\cite[p. 56]{acgh}). 
\end{proof}
 \begin{remark}
If $(C, \eta)$ has genus $2$, then the Prym-canonical system $|\omega_C \otimes \eta|$ has base points as it has dimension $0$; this agrees with the fact that every $2$-torsion line bundle on a genus $2$ curve can be written as $\eta=\O_C(p-q)$ where $p$ and $q$ are Weierstrass points. In particular, $\Cliff_{\eta}(C)=0$ for every Prym curve $(C,\eta)$ of genus $2$. 
 \end{remark}

\subsection{The étale double cover and its invariant Clifford index}
We end this section introducing a further invariant that encodes the relation between the Prym-canonical Clifford index of a Prym curve and its étale double cover.
\\Let $\pi \colon \tilde{C}\to C$ be the étale double cover induced by $\eta$ and let $\iota$ be the covering involution on $\tilde{C}$. The $\iota$-invariant line bundles on $\tilde{C}$ are parametrized by \begin{align*}
    \Pic(\tilde{C})^{\iota}: & = \{ \tilde{L} \in \Pic(\tilde{C}) \ | \ \iota^* \tilde{L} \simeq \tilde{L} \}= \\
    & = \{ \pi^*L \in \Pic(\tilde{C} \ ), \ L \in \Pic(C) \}.
\end{align*}
We introduce the following definition: \begin{definition}
    A line bundle $\tilde{L}\in \Pic(\tilde{C})$ \emph{contributes to the $\iota$-invariant Clifford index of $\tilde{C}$} if it is $\iota$-invariant and contributes to the Clifford index of $\tilde{C}$.
    
    The \emph{$\boldsymbol{\iota}$-invariant Clifford index of $\tilde{C}$} is \[
    \Cliff(\tilde{C})^{\iota}:= \mathrm{min} \Big \{ \Cliff(\tilde{L}) \ \Big | \ \tilde{L}\in \Pic(\tilde{C})^{\iota}, \ \deg(\tilde{L})\le g(\tilde{C})-1, h^0(\tilde{L})\ge 2 \Big \}.
    \]
\end{definition}
By the push-pull formula, a $\iota$-invariant line bundle $\tilde{L}= \pi^*L$ satisfies $\deg\tilde L=2\deg L$ and \[
h^0(\pi^*L)= h^0(L)+h^0(L \otimes \eta),
\] and thus $\tilde{L}= \pi^*L$ contributes to the $\iota$-invariant Clifford index of $\tilde{C}$ if and only if one of the following occurs: \begin{enumerate}
    \item[(a)] $L$ contributes to the Prym-canonical Clifford index of $(C, \eta)$;
    \item[(b)] one between $L$ and $L \otimes \eta$ contributes to the Clifford index of $C$.
\end{enumerate}
The next result relates the Prym-canonical Clifford index of $C$ and the $\iota$-invariant Clifford index of $\tilde{C}$.
\begin{proposition} \label{quasi}
    Let $(C, \eta)$ be a Prym curve of genus $g$ and $\pi \colon \tilde{C}\to C$ be the étale double cover induced by $\eta$, with covering involution $\iota$ on $\tilde{C}$. Denoting by $\mathrm{gon}(C)$ the gonality of $C$, one has: \[
    \Cliff(\tilde{C})^{\iota}= 2\mathrm{min} \{\Cliff_{\eta}(C),  \mathrm{gon}(C)-1 \}.
  \]
\end{proposition}
\begin{proof}
    Let $\tilde{L}= \pi^*L$ be a line bundle contributing to the $\iota$-invariant Clifford index of $\tilde{C}$. Without loss of generality, we can assume $h^0(L)\ge h^0(L \otimes \eta)$. If $L$ contributes to the Prym-canonical Clifford index of $(C, \eta)$, or equivalently, $h^0(L \otimes \eta)\ne 0$, one obtains \begin{equation*} \label{daje}
    \Cliff(\tilde{L})=2 \Cliff_{\eta}(L).
    \end{equation*} On the other hand, if $h^0(L \otimes \eta)=0$, then $L$ contributes to the Clifford index of $C$ and \begin{equation}\label{nuova}
    \Cliff(\tilde{L})=2\deg(L)-2h^0(L)+2= 2(\Cliff(L)+r(L)) \ge 2 (\Cliff(C)+1).
    \end{equation} If $C$ has Clifford dimension $1$, the lower bound $ 2(\Cliff(C)+1)$ is reached by choosing a pencil computing the gonality, and this finishes the proof since $\Cliff(C)=\mathrm{gon}(C)-2$. We now assume that $C$ has Clifford dimension $>1$, so that the inequality in \eqref{nuova} is never an equality. By \cite{cm}, one has $\mathrm{gon}(C)=\Cliff(C)+3$; therefore, the minimal value of $\Cliff(L)+r(L)$ as $L$ varies between the line bundles contributing to the Clifford index of $C$ is again obtained when $L$ is a pencil computing the gonality and equals $\Cliff(L)+r(L)=\Cliff(C)+2=\mathrm{gon}(C)-1$. 
\end{proof}

\begin{remark}
Proposition \ref{quasi} provides an alternative proof of Prym--Clifford's Theorem. Indeed, if $\Cliff_{\eta}(C)=0$, then $\Cliff(\tilde{C})^{\iota}=0$. The latter vanishing is equivalent to $\tilde{C}$ being hyperelliptic with a $\iota$-invariant $g^1_2$. Such a $g^1_2$ equals $\pi^*\O_C(p)$ for some $p \in C$ such that $\O_C(p)\otimes \eta$ is effective, that is, $\eta= \O_C(p-q)$ for some point $q$ such that $2p\sim 2q$, or equivalently, $p$ and $q$ are ramification points of the same $g^1_2$ on $C$.
\end{remark}

\section{Classification of curves with low Prym-canonical Clifford index}\label{quattro}
\subsection{Background information on curves of low gonality}In characterizing Prym curves $(C, \eta)$ with $ \Cliff_{\eta}(C)=1,2$, we will make use of some known results concerning linear series on trigonal, tetragonal and $5$-gonal curves, that we will here recollect.

We start by recalling the definition of the so-called Maroni invariant of a trigonal curve $C$. If $\pi \colon C \rightarrow \P^1$ is the $3-$sheeted covering induced by the $g^1_3$, then $\pi_*\omega_C=\O_{\P^1}(e_1)\oplus \O_{\P^1}(e_2)\oplus \O_{\P^1}(e_3)$ for some integers $e_1 \ge e_2 \ge e_3$, and $H^0(C, \omega_C \otimes kg^1_3)\simeq H^0(\P^1, \pi_*\omega_C\otimes \O_{\P^1}(k))\simeq \bigoplus_{i=1}^3 H^0(\P^1, \O_{\P^1}(e_i+k))$ (see \cite[p.172]{ms}). Thus, the integer $e_1$ is the maximum $r$ such that $\omega_C\otimes (rg^1_3)^{\vee}$ is effective. The Maroni invariant of $C$ can be defined as the integer  \begin{equation}\label{mardef}m:=g-2-e_1\end{equation} (see \cite[1.1, p. 172]{ms}). It satisfies the following (cf. \cite[(1.1)]{ms} and \cite[p.96]{bdc}): 
\begin{equation}\label{marin}\frac{g-4}{3}\le m \le \frac{g-2}{2},\,\,\,\,\,m-g \equiv 0 \ \mathrm{mod} \ 2.
\end{equation}

As concerns the Brill--Noether theory of a trigonal curve $C$, we define the varieties \[
U^r_d:= \left \{ 
\begin{aligned}
     & rg^1_3+D_{d-3r}, \ \ \ \mathrm{if}\ d-3r\ge0 \\ & \emptyset \ \ \ \ \ \ \ \ \ \ \ \ \ \ \ \ \ \ \ \ \ \ \   \ \mathrm{otherwise}
\end{aligned}
 \right \} \]
 \[V^r_d:= \left \{
\begin{aligned}
    & \omega_C-((g-d+r-1)g^1_3+  D_{2d-g-3r+1}), \ \ \  \mathrm{if} \ 2d-g-3r+1\ge0 \\
    & \emptyset, \ \ \ \ \ \ \ \ \ \ \ \ \ \ \ \ \ \ \ \ \ \ \ \ \ \ \ \ \ \ \ \ \ \ \ \ \ \ \ \ \ \ \ \ \ \ \ \ \ \ \ \ \ \ \ \ \ \ \ \ \ \ \ \ \  \mathrm{otherwise}
\end{aligned}
\right \}.
\] 

\begin{proposition}\cite[Prop. 1, p.173]{ms} \label{maroni}
    Let $C$ be a trigonal curve of genus $g$. For $d \le g-1$ and $r \ge 1$ we have: \begin{enumerate}
    \item $W^r_d(C)=U^r_d \cup V^r_d$;
    \item if $U^r_d \ne \emptyset$, then it is an irreducible component of $W^r_d(C)$;
    \item Assume $V^r_d \ne \emptyset$. Then also $U^r_d \ne \emptyset$, and $V^r_d$ is an irreducible component of $W^r_d$ different from $U^r_d$ if and only if $g-d+r-1 \le m$, where $m$ is the Maroni invariant.
    \end{enumerate}
\end{proposition}
\begin{corollary} \cite[Cor. 2, p. 175]{ms} \label{trig}
Let $C$ be a trigonal curve of genus $g$ and $L$ a line bundle on $C$ of degree $0 \le d \le g-1$. Then $h^0(L)\le \frac{d}{3}+1$, and equality holds if and only if either $L=\frac{d}{3} g^1_3$, or $d=g-1$ and $L=\omega_C-(\frac{d}{3} g^1_3)$. Furthermore, in the latter case we have $\omega_C - (\frac{d}{3} g^1_3)=\frac{d}{3} g^1_3$ if and only if the Maroni invariant $m$ is minimal (that is, $m= \frac{g-4}{3}$).
\end{corollary}

Linear series on a tetragonal curve can be described in terms of a fixed $g^1_4$ as follows.
\begin{corollary} \cite[Cor. 1.10]{cm00} \label{quadr}
Let $C$ be a tetragonal curve. Fix a $g^1_4$ on $C$. Then any base-point free $g^r_d$ on $C$ such that both $r \ge 1$ and $r':=g-d+r-1\ge 1$ is with respect to the fixed $g^1_4$ of one of the following types: \begin{enumerate}
\item $g^r_d= r g^1_4$;
\item the residual of the $g^r_d$ is  $r'g^1_4$ plus an effective divisor $F$;
\item $g^r_d=(r-1)g^1_4+F$, for some effective divisor $F$; \label{abc}
\item $g^r_d= l + kg^1_4$, with $l$ of type \eqref{abc} and $\dim l= r-2k{\geq 1}$.
\end{enumerate}
\end{corollary}
 Analogously, having fixed a $g^1_5$ on a curve of gonality $5$, its linear series are of the following types.
\begin{lemma}\cite[Lemma 1]{park} \label{5gonal}
  Let $C$ be a $5-$gonal curve. Fix a $g^1_5$. For any base point free $g^r_d$ on $C$ with both $r\geq 1$ and $r':=r-d+g-1\geq 1$, denoting by $B$ the base locus of $|\omega_C-g^r_d|$, one of the following holds with respect to the fixed $g^1_5$:  \begin{enumerate}
      \item $g^r_d=rg^1_5$;
      \item $g^r_d=(r-1)g^1_5+E$ for some effective divisor $E$ with $\dim (r-2)g^1_5+E=r-2$;
      \item $g^r_d-g^1_5$ has dimension $r-2\geq 1$ and moreover it is non-trivial, that is, it is base point free with both $h^0\ge 2$ and $h^1\geq 2$;
      \item $\omega_C-g^r_d-B=r'g^1_5$;
      \item $\omega_C-g^r_d-B=(r'-1)g^1_5+E$ for some effective divisor $E$ with $\dim (r'-2)g^1_5+E=r'-2$;
      \item $\omega_C-g^r_d-B-g^1_5$ has dimension $r'-2\geq 1$ and moreover it is non-trivial.
  \end{enumerate}  
\end{lemma}

\subsection{Prym-canonical Clifford dimension of bielliptic curves}
In our classification we will often apply the results by Coppens-Martens on linear series computing the Clifford index of a curve which is neither hyperelliptic nor bielliptic (cf. Section \ref{preliminary}). Hyperelliptic and bielliptic curves thus require a separate treatment. We postpone to the next section the case of hyperelliptic curves and we now concentrate on bielliptic curves.

We recall the following result on the Brill-Noether theory of bielliptic curves due to Coppens-Keem-Martens (cf. \cite{SP01} for a nice proof of the existence of a base point free $g^1_{g-1}$):
\begin{corollary}\cite[Cor. 2.2.2]{ckm}\label{bielliptic}Let $C$ be a curve of genus $g$ admitting a degree $2$ morphism $f \colon C \rightarrow E$ to an elliptic curve $E$. If $d\le g-1$, any base point free $g^r_d$ on $C$ is the pullback under $f$ of a $g^r_{d/2}$ on $E$, unless $d= g-1$ and $r=1$.
\end{corollary}
We are now ready to prove the following: 
\begin{theorem}\label{elliptic}
If $C$ is a bielliptic curve, then $\dim \Cliff_{\eta}(C)=(0,0)$.
\end{theorem}
\begin{proof}
      We denote by $\pi \colon C \rightarrow E$ the degree $2$ map to an elliptic curve $E$.\\ Assume that $\dim \Cliff_{\eta}(C)\neq(0,0)$ and let $L \in \Pic^d(C)$ be a line bundle of degree $d$ that computes the Prym-canonical Clifford dimension of $C$. Lemma \ref{lemma3} implies that both $L$ and $L\otimes \eta$ are base point free and, if either $d\le g-2$, or $d=g-1$ and $r(L)\ge 2$, Corollary \ref{bielliptic} yields $L=\pi^*l$ and $L \otimes \eta= \pi^*l'$, where $l,l'$ are complete $g^{f-1}_f$ on $E$ with $f=\frac{d}{2}$. Hence, $\dim \Cliff_{\eta}(C)=(f-1,f-1)$ and $\Cliff_{\eta}(C)=\Cliff_{\eta}(L)=2f-f-f+1=1$. As far as $f\ge2$, one can consider an effective divisor $D_k$ of degree $ k < f$ on $E$, and take $m=l(-D_k)$, $m'=l'(-D_k)$, so that $m$ and $m'$ are complete $g^{f-k-1}_{f-k}$ on $E$. We get $\eta=\pi^*(m-m')$ and $\Cliff_{\eta}(\pi^*m)=1=\Cliff_{\eta}(C)$, in contradiction with the assumption that $L$ computes the Prym-canonical Clifford dimension.
     We conclude that $d=g-1$ and $r(L)=1$; by Proposition \ref{lemma4}, $r(L\otimes \eta)\geq 2$, so that $L \otimes \eta= \pi^*l$, where $l$ is a complete $g^{f-1}_{f}$ on $E$ with $f=(g-1)/2\geq 3$. However, the equality $L^{\otimes 2}= \pi^*l^{\otimes 2}$ yields a contradiction, as well. Indeed, it implies that for every divisor $D=p_1+\cdots+p_{g-1}\in |L|$ the divisor $2D$ is invariant under the covering involution $\iota$. Since $L$ is not a pullback, then $D$ is not invariant under $\iota$  and thus at least one of the points $p_i$ is a ramification point of $\pi$; moving $D\in |L|$, we get infinitely many ramification points which is clearly a contradiction.
\end{proof}

\subsection{Classification theorems}
We are now ready to classify Prym-canonical curves having Prym-canonical Clifford index equal to $1$.

\begin{proof}[Proof of Theorem \ref{thm1}]
If $\omega_C\otimes \eta$ is not very ample, then $\Cliff_{\eta}(C)=1$ by Example \ref{ex2}.

Viceversa, assume $\Cliff_{\eta}(C)=1$ and let $L$ be a line bundle computing the Prym-canonical Clifford dimension of $(C,\eta)$ with $\deg(L)\le g-1$. 

If $\dim \Cliff_{\eta}(C)=(0,0)$, that is, $r(L)=r(L\otimes \eta)=0$, then $\deg(L)=2$. Denoting by $D$ the only effective divisor in $|L|$, the Riemann-Roch Theorem yields \[
h^0(\omega_C \otimes \eta (-D))= h^0(L \otimes \eta)-2+g-1= g-2,
\] that is, $\varphi_{\omega_C\otimes\eta}$ identifies the points in $D$. In particular, $\omega_C\otimes \eta$ is not very ample and the statement follows from Lemma \ref{first lemma}.

Therefore, we want to rule out the cases with $r(L\otimes \eta)\ge r(L)\ge 1$, where Lemma \ref{lemma2} yields \[\Cliff(L)+\Cliff(L\otimes \eta)=4,\]
and thus $0\le \Cliff(L\otimes \eta)\le 2$ as $\Cliff(L\otimes \eta)\le \Cliff(L)$. In particular, one has $0\le \Cliff(C)\le 2$ and, recalling that on a hyperelliptic curve of genus $g \ge 2$ any special linear series $g^r_d$ is a sum of $r$ copies of the $g^1_2$ plus $d-2r$ base points  (\cite[p. 41]{acgh}), the case $\Cliff(C)= 0$ can be excluded because it would force $L$ to have base points, in contradiction with Lemma \ref{lemma3}. Therefore, from now on we will assume $1\le\Cliff(C)\le 2$. Also note that, by Theorem \ref{elliptic}, the inequality $r(L\otimes\eta)\geq 1$ implies that $C$ is not bielliptic.\\\vspace{0.2cm}

\noindent\underline{CASE A: $\Cliff(C)=\Cliff(L\otimes \eta)=1$ and $\Cliff(L)=3$.} 

By Remark \ref{cliffbassi}, $L\otimes\eta$ is either a $g^1_3$, or a very ample $g^2_5$. The former case can be excluded because the inequalities $h^0(L)\ge 2$ and $h^0(L \otimes \eta) \ge 2$ imply  \begin{equation} \label{dis grado}
1=\Cliff_{\eta}(L)= \deg(L)-h^0(L)-h^0(L \otimes \eta)+1 \le \deg(L) -3.
\end{equation}
If $L\otimes \eta$ is a very ample $g^2_5$, then $C$ has genus $6$ and the equality $\Cliff_{\eta}(L)=1$ reads like $2=h^0(L)$ and so $h^0(\omega_C\otimes L^\vee)=2$ by Riemann-Roch. Take $D=p_1+\ldots+ p_5\in |\omega_C\otimes L^\vee|$; since $h^0(\omega_C(-D))=2$, the points $p_1, \ldots, p_5$ impose one condition less than expected to the canonical linear system $|\omega_C|$. Recalling that $|\omega_C|$ is cut out by plane conics, there is a one-dimensional family $\mathcal{C}=|I_{p_1+\cdots +p_5/\P^2}(2)|$ of conics passing through the five points $p_1, \ldots, p_5$ and any $Q \in \mathcal{C}$ is the union of two lines $r$ and $r'$. Since $\mathcal C$ is $1$-dimensional, one of the two lines (let us say $r'$) moves and thus contains at most one of the $p_i$. We first exclude the fact that four points, for example $p_1, \ldots, p_4$, lie in $r$ and $p_5 \in r'$ by showing that in this case $|L|$ would have a base point in contradiction with Lemma \ref{lemma3}. Indeed, any conic $Q\in\mathcal{C}$ would be the union of a fixed $r$ and a varying $r'$ through $p_5$. Any such $Q$ would cut out on the curve $5$ more points $q_1, \ldots, q_5$ and, up to reordering, we have that $q_1$ is the other intersection point of $r$ with $C$ and is thus fixed, while $q_2, \ldots, q_5 \in r'$ vary while varying $Q\in \mathcal C$ and thus define a $g^1_4$. Hence, the linear system $|L|=|\omega_C-p_1-\ldots-p_5|$ would have a base point at $q_1$. It remains to treat the case where $p_1 \ldots, p_5$ lie on the same line, but this would bring to the contradiction $L\simeq\O_{C}(1)\simeq L\otimes \eta$.\\\vspace{0.2cm}

\noindent\underline{CASE B: $\Cliff(C)=1$ and $\Cliff(L)=\Cliff(L \otimes \eta)=2$.} 

We first remark that $C$ cannot be a smooth plane quintic, because otherwise it would have genus $g=6$ and so any linear series of degree $\le g-1$ with Clifford index $2$ would be a $g^1_4$; however, it is not possible that both $L$ and $L \otimes \eta$ are $g^1_4$ by Lemma \ref{lemma4}. Therefore, the curve $C$ must be trigonal. By Corollary \ref{trig}, any line bundle of Clifford index $2$ on a trigonal curve is either a $g^1_4$, or a $g^2_6$. We exclude the former case as it would imply $\dim \Cliff_{\eta}(C)=(1,1)$, in contradiction with Proposition \ref{lemma4}. In the latter case, Corollary \ref{trig} implies that the only $g^2_6$ on $C$ are the linear series $2g^1_3$ and, if $g=7$, the residual to $2g^1_3$. Since $L$ and $L \otimes \eta$ are distinct as $\eta$ is non-trivial, we conclude that $g=7$ and $L = 2g^1_3$ and $L \otimes \eta= \omega_C-2g^1_3$, or viceversa. However, this yields the contradiction $\omega_C\otimes \eta=4g^1_3=\omega_C$, where the second equality again follows from Corollary \ref{trig} and the fact that the Maroni invariant $m$ of a trigonal genus $7$ curve is $1$ by \eqref{marin}.\\\vspace{0.2cm}

\noindent\underline{CASE C: $\Cliff(C)=\Cliff(L\otimes \eta)=\Cliff(L)=2$.}

The line bundles $L$ and $L\otimes \eta$ define either two $g^1_4$, or two $g^2_6$, or two $g^3_8$ and the genus is $9$. The former case cannot occur by Proposition \ref{lemma4}. In the latter case both the $g^3_8$ are simple (that is, they define a birational morphism onto the image) as $C$ is neither bielliptic nor hyperelliptic; by \cite[Thm 2.3]{lm}, we reach the contradiction $L\simeq L\otimes \eta$. Hence, $L$ and $L\otimes \eta$ are $g^2_6$ and they are simple (as $C$ is not hyperelliptic, bielliptic or trigonal) but not very ample because of the unicity of a $g^2_d$ on a smooth plane curve of degree $d \ge 4$ (\cite[p. 56]{acgh}). 
Hence, we can assume that $L\otimes \eta$ is a $g^2_6$ defining a morphism $f \colon C \rightarrow \bar{C}\subseteq \P^2$ that is generically injective but not embedding. Since plane sextics have arithmetic genus $10$ and $g\geq d+1=7$, the $\delta$-invariant of $\bar{C}$ satisfies $1\le\delta(\bar{C})\le 3$. We denote by $A_{\bar C}$ the conductor ideal of the normalization $\nu:C\rightarrow \bar{C}$ and by $E_{\bar C}\subset \bar{C}$ the conductor subscheme, which is defined by $A_{\bar C}$ and has length equal to $\delta(\bar C)$. Since $A_{\bar C}$ is also an ideal sheaf of $\nu_*\mathcal O_C$, there exists an effective divisor $\Delta_C$ on $C$ such that $\deg\,\Delta_C=2\delta(C)$ and $A_{\bar C}=\nu_*\mathcal O_C(-\Delta_C)$. The isomorphism 
\begin{equation}\label{can}\omega_C\simeq \nu^*\omega_{\bar C}\otimes \mathcal O_C(-\Delta_C),\end{equation} along with the fact that the dualizing sheaf $\omega_{\bar C}$ is cut out by cubics by adjunction, yields that the sections of $\omega_{ C}$ are cut out by cubics through $E_{\bar C}$. We now fix a general element $D=p_1+\ldots+p_6 \in |L|$; Riemann-Roch Theorem yields $h^0(\omega_C-p_1-\cdots-p_6)=6-\delta(\bar C)$. We set $\bar D:=E_{\bar C}+p_1+ \ldots+p_6$, which is a divisor on $\bar C$ consisting of $6+\delta(\bar C)$ possibly infinitely near points of $\mathbb P^2$. Then, there exists a family $\mathcal{D}_{6+\delta(\bar C)}$ of plane cubics passing through $\bar D$ of dimension 
\begin{equation}\label{dimDi}
\dim\,\mathcal{D}_{6+\delta(\bar C)}=5-\delta(\bar C).
\end{equation}This excludes the case $\delta(\bar C)=3$ because there exists at most a pencil of plane cubics through a $0$-dimensional subscheme $\xi\subset \mathbb P^2$ of length $9$. Indeed, let $X_1$, $X_2$ be plane cubics through such a $\xi $. The short exact sequence \[
0 \rightarrow \mathcal{I}_{X_1/\P^2}(3)\rightarrow \mathcal{I}_{\xi/\P^2}(3)\rightarrow \mathcal{I}_{\xi/X_1}(3)\rightarrow0
\] reads like \[
0 \rightarrow \O_{\P^2} \rightarrow \mathcal{I}_{\xi/\P^2}(3) \rightarrow \O_{X_1} \rightarrow 0,
\] hence $h^0(\mathcal{I}_{\xi/\P^2}(3))=2$.

Therefore, we must have $1\le \delta(\bar C)\le 2$ and, by \eqref{dimDi} and \cite[Cor. 4.4, ch. V]{hartshorne}, the divisor $\bar D=E_{\bar C}+p_1+ \ldots+p_6$ splits as a sum $\bar D=\bar{D}_1+\bar{D}_2$ of two generalized effective divisors $\bar{D}_1,\bar{D}_2\subset \bar C$ such that either $\bar D_1$ is contained in a line and has degree $\ge 4$, or $\bar D_2$ lies on a conic $Q$ and has degree $\ge 7$. The latter case can be excluded because such a conic $Q$ would be fixed and thus the family $\mathcal{D}_{6+\delta(\bar C)}$ would contain only reducible curves of the form $Q\cap l_t$ where $l_t$ is a line; in particular, $\mathcal{D}_{6+\delta(\bar C)}$ would have dimension $\le 2$ in contradiction with \eqref{dimDi}. Hence, $\bar D_1$ is contained in a fixed line $l$ and has degree $4\le d_1\le 6$ (where the right hand inequality follows from the fact that $\bar C$ is a sextic), and every cubic in the family $\mathcal{D}_{6+\delta(\bar C)}$ is the union of $l$ with a conic $Q_t\in |I_{\bar D_2/\mathbb P^2}(2)|$. Note  that $\bar D_2$ has degree \begin{equation}\label{d2}\delta(\bar C)\le d_2\le 2+\delta(\bar C)\end{equation} and \eqref{dimDi} can be rewritten as \begin{equation}\label{dimD1}\dim |I_{\bar D_2/\mathbb P^2}(2)|=5-\delta(\bar C).\end{equation} 

If $\delta(\bar C)=1$, equations \eqref{d2} and \eqref{dimD1} yield $d_2=1$, $d_1=6$ and $D_2=E_{\bar C}$ because $\bar C$ is singular at $E_{\bar C}$ and otherwise the intersection between $\bar C$ and $l$ would be $\ge 7$, which is absurd. Analogously, if $\delta(\bar C)=2$, equations \eqref{d2} and \eqref{dimD1} yield $d_2=2$, $d_1=6$ and thus again $D_2=E_{\bar C}$. It then follows that $\omega_C( - D) \simeq \nu^*\mathcal O_{\bar C} (2)\otimes \mathcal O_C(-\Delta_C)$, that together with \eqref{can} yields $L\simeq \mathcal O_C(D) \simeq \O_C(1) \simeq L \otimes \eta$, yielding a contradiction.

\end{proof}

\begin{example}\label{dom}
Equality $\Cliff_\eta(C)=1$ occurs for instance when $C$ is hyperelliptic and $\eta=\O_C(w_1+w_2-w_3-w_4)$ with $w_1,\ldots,w_4$ Weierstrass points (so that $2(w_1+w_2)\sim 2(w_3+w_4)$).
\end{example}

\begin{remark}
A smooth nonhyperelliptic Prym curve $(C, \eta)$ of genus $g=3$ such that $\omega_C \otimes \eta$ is base point free has Prym-canonical Clifford index $\Cliff_{\eta}(C)=1$. Indeed, the Prym-canonical system $\omega_C \otimes \eta \sim g^1_4$ is not very ample.
\end{remark}
As a consequence of Theorems \ref{clifford thm} and \ref{thm1}, one gets the following: 
\begin{corollary} \label{veryampleness}

Let $(C, \eta)$ be a Prym curve of genus $g \ge 3$. Then, $\omega_C \otimes \eta$ is very ample if and only if $\Cliff_{\eta}(C)\ge 2$.

\end{corollary}
Let us now move on to the proof of the classification theorem for Prym curves with Prym-Clifford dimension $2$.

\begin{proof}[Proof of Theorem \ref{thm2}]
If $D$ is the divisor cut out by a trisecant line, then $h^0(\O_C(D))=1$ and $h^0(\omega_C\otimes \eta (-D))=g-3$, so that $\Cliff_{\eta}(\O_C(D))=2$.

Viceversa, assume $\Cliff_{\eta}(C)=2$ and let $L$ be a line bundle computing the Prym-canonical Clifford dimension of $(C,\eta)$.

If $\dim \Cliff_{\eta}(C)=(0,0)$, that is, $r(L)=r(L\otimes \eta)=0$, then $\deg(L)=3$. Riemann-Roch Theorem then implies that the divisor $D\in |L|$ imposes one condition less than expected to the hyperplane sections of the Prym-canonical curve, or equivalently, $D$ is cut out by a trisecant line to $\varphi_{\omega_C\otimes\eta}(C)$. 

We thus want to rule out the cases $2\le h^0(L)\le h^0(L \otimes \eta)$, where $C$ is not bielliptic by Theorem \ref{elliptic} and we have $$\Cliff(L)+\Cliff(L \otimes \eta)=6,\,\,\,\Cliff(L\otimes\eta)\le \Cliff(L)$$ by Lemma \ref{lemma2}. In particular, we obtain $0\leq\Cliff(C)\leq 3$. When $C$ is hyperelliptic, as in the proof of Theorem \ref{thm1} we conclude that $L $ has some base points, in contradiction with Lemma \ref{lemma3}. Therefore, we can assume $1\le\Cliff(C)\le 3$. 
\\We first consider the cases where $\Cliff(C)=1$ and thus $C$ is either trigonal or a smooth plane quintic.\\\vspace{0.2cm}

\noindent\underline{CASE A: $\Cliff(C)=\Cliff(L\otimes\eta)=1$, and $\Cliff(L)=5$.} 
Since $L\otimes \eta$ is either a $g^1_3$ or a $g^2_5$, the degree of $L$ is either $3$ or $5$ and in both cases the equality $\Cliff(L)=5$ yields a contradiction.\\\vspace{0.2cm}

\noindent\underline{CASE B: $\Cliff(C)=1$, $\Cliff(L \otimes \eta)=2$ and $\Cliff(L)=4$.}  

In this case $C$ cannot be a smooth plane quintic because otherwise its genus would be $6$ in contradiction with the assumptions $d\leq g-1$ and $\Cliff(L)=4$. Hence, $C$ is trigonal. By Corollary \ref{trig}, a $g^r_d$ with Clifford index $2$ satisfies $r \le 2$, and this implies that $L \otimes \eta$ is either a $g^1_4$ or a $g^2_6$. If $L \otimes \eta$ is a $g^1_4$, then $r(L)=0$ and $\dim \Cliff_{\eta}(C)=(0,1)$, in contradiction with Proposition \ref{lemma1}. If $L \otimes \eta$ is a $g^2_6$, then $L$ is a $g^1_6$. Assume $g=7$, so that the Maroni invariant of $C$ is $m=1$ by \eqref{marin} and is thus minimal. By Corollary \ref{trig} we get $L \otimes \eta= \omega_C - 2g^1_3=2g^1_3$ and $\omega_C=4 g^1_3$. Since $L$ is base-point free by Lemma \ref{lemma3}, then $L \in V^1_6$, and so $L=\omega_C-g^1_3-D_3=3g^1_3-D_3$ for some degree $3$ effective divisor $D_3$. This yields that $\eta=g^1_3-D_3$, and so $\Cliff_{\eta}(C)\le \Cliff_{\eta}(g^1_3)=1$, in contradiction with our hypothesis. Therefore, by Corollary \ref{trig} we can assume $g\ge 8$ and $L \otimes \eta= 2g^1_3$. Since $L\in V^1_6\neq \emptyset$, we must have $g \le 10$.  If $g=8$ (respectively, $g=9$), the Maroni invariant is $m=2$ (resp., $m=3$). In both cases \eqref{mardef} yields $e_1=4$ and thus $\omega_C-4g^1_3$ is effective, that is, $\omega_C=4g^1_3\otimes \O_C(D_{2g-14})$ for some effective divisor $D_{2g-14}$ of degree $2g-14$. Since $L \in V^1_6$, there is an effective divisor $E_{10-g}$ of degree $10-g$ such that $L=\omega_C-(g-6)g^1_3-E_{10-g}$. In particular, for $g=8$ one gets that $\eta=L-(L \otimes \eta)=\omega_C-2g^1_3-E_2-2g^1_3= D_2-E_2$, and thus the contradiction $\Cliff_{\eta}(C)\le \Cliff_{\eta}(\O_C(D_2))=1$. If $g=9$, we get $\eta= \omega_C-3g^1_3-E_1-2g^1_3=D_4-(g^1_3+E_1)$ and, having fixed a point $p$ in the support of $D_4$, we can write  $D_4=D_3+p$ and $g^1_3(-p)=\O_C(E_2)$ for some effective divisors $D_3,E_2$ of degree $3$ and $2$, respectively, so that $\eta=\O_C(D_3-E_2-E_1)$. We get $\Cliff_{\eta}(\O_C(D_3))=2=\Cliff(C)$ and thus the contradiction $\dim \Cliff_{\eta}(C)=(0,0)$. To exclude the case $g=10$, it suffices to remark that $L\in V^1_6$ implies $L=\omega_C-4g^1_3=\omega_C\otimes L^{-2}$ so that $L^{\otimes 2}$ is special; this contradicts \cite[Cor. 3]{ms}.\\\vspace{0.2cm}

\noindent\underline{CASE C: $\Cliff(C)=1$ and $\Cliff(L)=\Cliff(L \otimes \eta)=3$}  

Both $L$ and $L\otimes \eta$ are $g^r_{2r+3}$ and we can assume $r\ge 2$, as $r=1$ contradicts Proposition \ref{lemma4}. Because of the inequality $d\leq g-1$, the curve $C$ cannot be a smooth plane quintic and is thus trigonal.
Again by Corollary \ref{trig}, we have $r \le 3$ and so $L$ and $L\otimes \eta$ are either of type $g^2_7$, or of type $g^3_9$ when $g=10$. 
In the former case, since $L$ and $L\otimes \eta$ are base point free by Lemma \ref{lemma3}, they cannot lie in  $U^2_7$. Hence, $V^2_7$ is nonempty (and thus $g \le 9$) and we have $L= \omega_C-(g-6)g^1_3-D_{9-g}$, $L \otimes \eta= \omega_C-(g-6)g^1_3-{D}_{9-g}'$ for some effective divisors $D_{9-g}$ and $D_{9-g}'$ of degree $9-g$. We obtain $\eta=\O_C(D_{9-g}-D_{9-g}')$, which is a contradiction because for $g=9$ it gives $\eta=\mathcal O_C$ and for $g=8$ it implies $\Cliff_{\eta}(C)\leq \Cliff_{\eta}(\mathcal O_C(D_{1}))=0$.
\\ If $L$ and $L \otimes \eta$ are two distinct $g^3_9$, Corollary \ref{trig} implies $g=10$, $L=3g^1_3$, $L \otimes \eta=\omega_C-3g^1_3$ (or viceversa) and the Maroni invariant is not minimal, that is, $m=4$. Hence, $e_1=g-2-m=4$ and $\omega_C(-4g^1_3)$ is effective; more precisely, if $\pi:C\to\mathbb P^1$ is the degree $3$ cover, using \cite[p.172]{ms} one computes $\pi_*\omega_C=\O_{\P^1}(4)\oplus \O_{\P^1}(4)\oplus \O_{\P^1}(-2)$ and thus $h^0(\omega_C(-4g^1_3))=2$, that is, $M_6:=\omega_C(-4g^1_3)$ is a $g^1_6$. We have $\eta=\omega_C-6g^1_3=M_6-2g^1_3$ and $\Cliff_{\eta}(M_6)=2=\Cliff_\eta(C)$, in contradiction with the assumption that $L$ computes the Prym-canonical Clifford dimension of $(C,\eta)$.\\\vspace{0.1cm}

We will now treat the cases where $\Cliff(C)=2$, that is, $C$ is either tetragonal or a smooth plane sextic (cf. Remark \ref{cliffbassi}).\\\vspace{0.2cm}

\noindent\underline{CASE D: $\Cliff(C)=\Cliff(L \otimes \eta)=2$ and $\Cliff(L)=4$.}

 Then $L\otimes \eta$ is either a $g^1_4$, or a $g^2_6$, or a $g^3_8$ and $g=9$. We exclude the case $L\otimes \eta= g^1_4$, as it would force $r(L)=0$, in contradiction with Proposition \ref{lemma1}. If $L \otimes \eta= g^2_6$, then $L= g^1_6$.

We may assume that $L \otimes \eta$ is simple (as $C$ is neither trigonal nor bielliptic) and thus defines a birational morphism $f \colon C \rightarrow \bar{C}\subseteq \P^2$. Since smooth plane sextics have arithmetic genus $10$ and $g \ge d+1=7$, the $\delta-$invariant of $\bar{C}$ satisfies $0 \le \delta(\bar{C})\le 3$. Using the same notation as in the proof of Theorem \ref{thm1}, we recall the isomorphism \eqref{can} implying that sections of $\omega_C$ are cut out by cubics through the conductor subscheme $E_{\bar C}$, which has length $\delta(\bar{C})$.
Given a general divisor $D=p_1+\cdots+p_6 \in |L|$, Riemann-Roch Theorem yields that $h^0(\omega_C(-D))=5-\delta(\bar{C})$, that is, the divisor $\bar{D}:= E_{\bar C}+D$ (consisting of $6+\delta(\bar{C})$ possibly infinitely near points of $\P^2$) imposes one condition less than expected to $|\O_{\P^2}(3)|$, and there is a family $\mathcal{D}_{6+\delta({\bar{C}})}$ of plane cubics passing through $\bar{D}$ of dimension \begin{equation}\label{dimD}
    \dim \mathcal{D}_{6+\delta({\bar{C}})}= 4- \delta(\bar{C}).
\end{equation}
We firstly consider the cases where $0\leq \delta(\bar{C}) \leq 2$. By Corollary \cite[Cor. 4.4, ch. V]{hartshorne}, the divisor $\bar{D}=E_{\bar C}+D$ splits as a sum $\bar{D}=\bar{D_1}+\bar{D_2}$ of two generalized effective divisors $\bar{D_1},\bar{D_2} \subset \bar{C}$ such that either $\bar{D_1}$ is contained in a line and has degree $\ge 4$, or $\bar{D_2}$ lies on a conic $Q$ and has degree $\ge 7$.\\ 
If $\bar{D_2}$ has degree $\ge 7$ and lies on a conic, then such a conic $Q$ would be fixed and thus the family $\mathcal{D}_{6+\delta(\bar C)}$ would contain only reducible curves of the form $Q\cup l_t$ where $l_t\in |\O_{\mathbb P^2}(1)|$ is a line, so that \eqref{dimD} yields $\delta(\bar C)=2$ and $\bar D=\bar D_2$. We can write $\omega_C\otimes L^\vee=\omega_C(-D)=\nu^*(\O_{\bar C}(1)+q_1+q_2)$, where $q_1,q_2$ are the two further intersection points of $Q$ with $\bar C$. Hence, the sections of $L=\nu^*(\O_{\bar C}(2)-E_{\bar C}-q_1-q_2)$ would be cut out by plane conics $\gamma_t$ through $E_{\bar C}+q_1+q_2$. Since $g=8$ and $C$ is not hyperelliptic, then $h^0(C,(L\otimes \eta)^{\otimes 2})=6$, that is, all sections of $|(L\otimes \eta)^{\otimes 2}|$ are cut out by plane conics. Since $L^{\otimes 2}=(L\otimes \eta)^{\otimes 2}$, the divisor $\nu^*(2E_{\bar C}+2q_1+2q_2)=2\Delta_C+\nu^{-1}(2q_1+2q_2)$ is the pullback to $C$ of the intersection of $\bar C$ with a conic; such a conic has to be singular along $E_{\bar C}$, which has length $2$, and is thus a double line. Hence, $E_{\bar C},q_1,q_2$ are collinear, yielding the contradiction $L=L\otimes \eta$.\\
We now assume that $\bar{D_1}$ has degree $d_1\geq 4$ and is contained in a fixed line $l$, so that $\omega_C(-D)$ has the further intersection points of $l$ with $\bar C$ as base points and its base point free part is cut out by $|\mathcal{I}_{\bar{D_2}/\P^2}(2)|$; equation \eqref{dimD} then implies that $\bar D_2$ has degree $d_2=\delta(\bar C)+1$ and thus $d_1=5$. We now write $E_{\bar C}=E_1+E_2$ with $\bar{D_1}=E_1+D_1$, $\bar{D_2}=E_2+D_2$, where all divisors $E_i,D_i$ are effective. Since $\bar C$ is singular along $E_1$ and $E_2$, the equality $d_1=5$ implies that $E_1$ has degree $e\leq 1$ (because otherwise the intersection number between $\bar C$ and $l$ would be $>6$, which is absurd). If $e=1$, then $E_2$ has degree $\delta(\bar C)-1$, the divisor $D_2$ has degree $2$ and $\omega_C\otimes L^\vee=\omega_C(-D)=\nu^*(\O_{\bar C}(2)(-D_2-E_2)$. We get $L=\nu^*(\O_{\bar C}(1)+D_2-E_1)$ so that $\eta=\nu^*\O_{\bar C}(D_2-E_1)=\O_C(D_2-\Delta_1)$ where $\Delta_1\subset C$ is the degree $2$ divisor mapping to $E_1$; we get the contradiction $\Cliff_{\eta}(C)\leq \Cliff_{\eta}(D_2)=1$. If instead $e=0$, then $E_2=E_{\bar C}$, the divisor $D_2$ has degree $1$ and $\omega_C\otimes L^\vee=\nu^*(\O_{\bar C}(2)(-D_2-E_2+q)$, where $q$ is a base point and coincides with the further intersection point of $\bar C$ with $l$. We obtain $L=\nu^*(\O_{\bar C}(1)+D_2-q)$ and thus $\eta=\O_C(D_2-q)$, yielding the contradiction $\Cliff_{\eta}(C)\leq \Cliff_{\eta}(D_2)=0$. 

It remains to treat the case $\delta(\bar{C})=3$, that is, $g=7$, where the divisors in
$|\omega_C(-D)|$ are cut out by the pencil of cubics through $\bar D$, which has degree $9$. The curve $\bar C$ has only double points as singularities because otherwise it would have a $g^1_3$. Let $\sigma:S\to\mathbb P^2$ be the blow-up of $3$ (possibly infinitely near) points $n_1,n_2,n_3$ resolving the singularities of $\bar C$, and denote by $l$ the class of the pullback on $S$ of a line so that $L\otimes \eta=\mathcal O_C(l)$. We recall that $\omega_S\simeq \sigma ^*\O_{\P^2}(-3)+E_1+E_2+E_3$, where $E_1,E_2, E_3$ are effective (possibly reducible) divisors contracted by $l$ such that $E_i^2=-1$ and $E_i\cdot E_j=0$ for $i\neq j$. Since $\bar C$ has only double points, then $C\in |6l-2E_1-2E_2-2E_3|$ and thus $\omega_C=\mathcal O_C(3l-E_1-E_2-E_3)$. The isomorphism $L^{\otimes 2}=(L\otimes \eta)^{\otimes 2}$ implies that, for every divisor $D\in |L|$, $\mathcal O_C(2l-2D)\simeq \mathcal O_C$. We denote by $2D$ the $0$-dimensional subscheme of $S$ of length $12$ consisting of the $6$ points of $D$ along with the tangent direction of $C$ at each of them, and consider the following short exact sequence
\begin{equation}\label{gennaio}
0\longrightarrow -4l+2E_1+2E_2+2E_3\longrightarrow 2l\otimes I_{2D/S}\longrightarrow \mathcal O_C\longrightarrow 0.
\end{equation}
An element of the linear system $|2l\otimes I_{2D/S}|$ defines a conic totally tangent to $\bar C$. Since a plane sextic with ADE singularities admits at most finitely many totally tangent conics \footnote{indeed, if $q:Y\to \mathbb P^2$ is a double cover of $\mathbb P^2$ branched along a sextic $X_6$ with ADE singularities, then $Y$ is a $K3$ surface with ADE singularities and the inverse image on $Y$ of every conic $\gamma$ totally tangent to $X_6$ splits as the union of two smooth rational curves. Since smooth rational curves on $Y$ are finitely many, the same holds for the conics $\gamma$.}, we conclude that $h^0(2l\otimes I_{2D/S})=0$ for a general $D\in |L|$. Short exact sequence \eqref{gennaio} then implies that $h^1(-4l+2E_1+2E_2+2E_3)\neq0$. Because of the equalities $$h^1(-4l+2E_1+2E_2+2E_3)=h^1(l-E_1-E_2-E_3)=h^0(l-E_1-E_2-E_3),$$
which follow from Serre duality and the Riemann-Roch Theorem, respectively, we conclude that $h^0(l-E_1-E_2-E_3)\neq 0$, that is, $n_1,n_2,n_3$ lie on a line. Therefore, $\omega_C=\O_C(2l)$ and both $L\otimes \eta$ and $L$ are theta characteristics of $C$. We can write $\omega_C\otimes \eta=\mathcal O_C(4l-E_1-E_2-E_3-D)$ for $D\in |L|$; since $$h^0((4l-E_1-E_2-E_3)\otimes I_{D/S})\geq \chi (4l-E_1-E_2-E_3)-6=6=h^0(\omega_C\otimes \eta),$$ we conclude that the divisors in $|\omega_C\otimes \eta|$ are cut out by plane quartics passing through $n_1,n_2,n_3$ and $D$. We want to show that $\eta$ can be written as the difference of two effective divisors of degree $3$, or equivalently, there exists $D_3\in \mathrm{Sym}^3(C)$ such that $h^0(\omega_C\otimes \eta(-D_3))= 4$. Choose a plane cubic $X\in |3l-E_1-E_2-E_3-D|$ and a divisor $D_3\subset X\cap \bar C$ with support disjoint from that of $E_1+E_2+E_3+D$, and consider the following short exact sequence:
$$0\longrightarrow  l \longrightarrow (4l-E_1-E_2-E_3)\otimes I_{D+D_3/S}\longrightarrow  \mathcal O_X(4l-E_1-E_2-E_3-D-D_3)\longrightarrow 0.$$
We observe that $\mathcal O_X(4l-E_1-E_2-E_3-D-D_3)\simeq \mathcal O_X(3l-D-D_3)\simeq O_X$, where the last isomorphism follows from the existence of a pencil of cubics through $D+D_3$. 
We thus obtain $h^0((4l-E_1-E_2-E_3)\otimes I_{D+D_3/S})=4$ and, by restricting to $C$, $h^0(\omega_C\otimes \eta(-D_3))= 4$. As a consequence, $\mathcal O_C(D_3)\otimes \eta$ is effective and thus the contradiction $\Cliff_\eta(C)\le \Cliff_\eta(\O_C(D_3))=1$.

We now treat the case where $g=9$ and $L \otimes \eta$ is a $g^3_8$, which implies that $L$ is a $g^2_8$ and $C$ is tetragonal. Note that $\omega_C\otimes \eta\otimes L^\vee$ is also a $g^3_8$ and $\omega_C\otimes L^\vee$ is a $g^2_8$, so that $\omega_C\otimes L^\vee$ also computes the Prym-canonical Clifford dimension of $(C,\eta)$. By Lemma \ref{lemma3} $L$, $L\otimes \eta$, $\omega_C\otimes L^\vee$, $\omega_C\otimes \eta\otimes L^\vee$ are all base-point free and thus we can use Corollary \ref{quadr} to describe them having fixed a $g^1_4$, that we call $l$. Since both $L\otimes \eta$ and $\omega_C\otimes \eta\otimes L^\vee$ are $g^3_8$, for degree reasons they must be of type (4) in Corollary \ref{quadr}, that is, $L\otimes \eta=l+l'$ and $\omega_C\otimes \eta\otimes L^\vee=l+l''$ where $l'$ and $l''$ are $g^1_4$ different from $l$. If we had $L=l+F_4$ with $F_4$ an effective divisor of  degree $4$, then $\eta=l'-F_4$ and the equality $ \Cliff_{\eta}(C)=\Cliff_{\eta}(F_4)$ would contradict the assumption that $L$ computes the Prym-canonical Clifford dimension of $C$. 
 Therefore, Corollary \ref{quadr} yields either $L= 2l$, or $\omega_C\otimes L^\vee=2l$. In the former case we get $\eta= l+l'-2l= l'-l$, while in the latter case we get $\eta=l+l''-2l=l''-l$; in both cases we thus obtain the contradiction $ \Cliff_{\eta}(C)\leq\Cliff_{\eta}(l)=1$.\\\vspace{0.2cm}

\noindent\underline{CASE E: $\Cliff(C)=2$ and $\Cliff(L\otimes\eta)=\Cliff(L)=3$.} 
 
Then both $L$ and $L \otimes \eta$ are $g^r_{2r+3}$, and Proposition \ref{lemma4} implies $r \ge 2$. Since $L$ computes the Prym-canonical Clifford dimension, then $L$ and $L\otimes \eta$ are base point free by Lemma \ref{lemma3}. Moreover, since $\Cliff(\omega_C\otimes L^\vee)=\Cliff(\omega_C\otimes \eta\otimes L^\vee)=\Cliff(C)+1$, then $\omega_C\otimes L^\vee$ and $\omega_C\otimes\eta\otimes  L^\vee$ have at most one base point (indeed, when subtracting $s$ base points to a linear series, the Clifford index decreases by $s$).\\Assume that $C$ is a smooth plane sextic, so that $g=10$, and $L, L \otimes \eta$ are two $g^2_7$, or two $g^3_9$. To rule out the $g^3_9$ case, it is enough to apply Lemma \ref{5gonal} (smooth plane sextics have gonality $5$) excluding the cases (3) and (6) in the lemma because they would imply the existence of a $g^1_4$. The $g^2_7$ case can be excluded as follows. Fix a divisor $D=p_1+\cdots+p_7 \in |L|$; since by Riemann-Roch $h^0(\omega_C-D)=5$ and in the plane model of $C$ as a smooth sextic $\omega_C$ is cut out by plane cubics by the adjunction formula, 
there exists a $4$-dimensional family $\mathcal{D}=|\mathcal{I}_{p_1+\cdots+p_7/\P^2}(3)|$ of cubics passing through $p_1, \ldots p_7$. Then, \cite[Cor. 4.4, ch. V]{hartshorne} implies that either $4$ of $p_1, \ldots, p_7$ are collinear, or $7$ of them lie on a conic. If the latter case happens, the conic is fixed, and thus the contradiction $\dim \mathcal{D}=\dim |\O_{\P^2}(1)|=2 < 4$. The former case implies that up to reordering, $p_1 \ldots, p_6$ lie on a line $l$ (as $C$ is a sextic) and $p_7 \in Q$, where $Q$ is a conic. Then, $\dim \mathcal{D}=4$, and $\omega_C \otimes L^{\vee} \simeq \O_C(2)-p_7$. This implies $L \simeq \O_C(1)+p_7$, that is, $L$ has a base point at $p_7$, in contradiction with Lemma \ref{lemma3}.

Therefore we can assume that $C$ is tetragonal and refer to Corollary \ref{quadr} to classify linear series of type $g^r_{2r+3}$ with $r\geq 2$ on it. Only cases (3) and (4) in the corollary may occur, as one excludes cases (1) and (2) for degree reasons recalling that $L$, $L\otimes \eta$ are base point free while their residuals have at most one base point. Case (3) reads like $g^r_{2r+3}=(r-1)g^1_4+F$ for some effective divisor $F$. On the other hand, case (4) reads like $g^r_{2r+3}=l+ kg^1_4$ with $l=g^{r-2k}_{2r+3-4k}$ of type (3), that is, $l=(r-2k-1)g^1_4+F$ and thus $g^r_{2r+3}=(r-k-1)g^1_4+F$. Hence, independently whether we fall in case (3) or (4), we get $L=(r-k-1)g^1_4+F$ and $L'=(r-k'-1)g^1_4+F'$ for some effective divisors $F,F'$ of degree $-2r+7+4k\leq 3$ and $-2r+7+4k'\leq 3$, respectively, where the inequalities follow from $r-2k\geq 1$ and $r-2k'\geq 1$; since both $\deg F$ and $\deg F'$ are congruent to $2r+3$ modulo $4$, we conclude that $\deg F=\deg F'$ (that is, $k=k'$). Using that $\eta\simeq \O_C(F-F')$ and $h^0(\O_C(F))=h^0(\O_C(F'))=1$ since $C$ has no $g^1_3$, we compute $\Cliff_\eta(\O_C(F))=\deg F-1\leq 2$ in contradiction with the assumption that $L$ computes the Prym-canonical Clifford dimension of $C$.\\\vspace{0.2cm}

\noindent\underline{CASE F: $\Cliff(C)=\Cliff(L)=\Cliff(L\otimes \eta)=3$.} 

By Remark \ref{cliffbassi}, $L$ and $L \otimes \eta$ are two $g^1_5$, or two $g^2_7$, or two $g^3_9$ and the genus is $10$. By Proposition \ref{lemma4}, we exclude the case where both $L$ and $L \otimes \eta$ are $g^1_5$. We stress that $C$ cannot have Clifford dimension $3$ because in this case the line bundle computing the Clifford index is unique (see \cite[p.174]{elms}). Therefore, $C$ is either $5$-gonal, or a smooth plane septic.

   Assume $C$ is a smooth plane septic. Then, $L$ and $L \otimes \eta$ cannot be two $g^3_9$ as they occur in genus $10$ while $C$ has genus $15$. If $L, L \otimes \eta$ are two $g^2_7$, then they have to be very ample as $C$ has genus $15$. The unicity of a $g^2_d$ on a smooth plane curve of degree $d\ge4$ excludes this case. 
   
   We thus conclude that $C$ has Clifford dimension $1$ and gonality $5$.
Since $L$ and $L \otimes \eta$ compute the Clifford index of $C$, they and their residuals are base-point free and we can apply Lemma \ref{5gonal} to classify them.\\ We first consider the case where $L, L \otimes \eta$ are two $g^2_7$. We exclude type (1) of the lemma for degree reasons and type (3) because it requires $r-2\geq 1$.  Type (4) reads like $\omega_C-g^2_7=(g-6)g^1_5$, and comparing the degrees we obtain $g=7$, against our assumption $g \ge d+1=8$. If a $g^2_7$ is of type (5), then $\omega_C-g^2_7-(g-7)g^1_5$ is effective and, for degree reasons, this yields $g=8$. In case (6), $\omega_C-g^2_7-g^1_5$ is a $g^{g-8}_{2g-14}$ with $g-8\geq 1$ and this yields the contradiction $\Cliff(g^{g-8}_{2g-14})=2<\Cliff(C)=3$. Hence, $L$ and $L\otimes \eta$ are of type (2), or of type (5) and $g=8$.\\If both $L$ and $L\otimes \eta$ are of type (2), then $L=g^1_5+E_2$ and $L\otimes \eta=g^1_5+E_2'$ for some effective degree $2$ divisors $E_2, E_2'$ satisfying $h^0(\O_C(E_2))=h^0(\O_C(E_2'))=1$; since $\eta=\O_C(E_2-E_2')$, this leads to the contradiction $\Cliff_{\eta}(\O_C(E_2))=1$. Similarly, if $g=8$ and both $L$ and $L \otimes \eta$ are of type (5), then $\omega_C\otimes L^\vee=g^1_5+E_2$ and $\omega_C\otimes \eta\otimes L^\vee=g^1_5+E_2'$ for some effective degree $2$ divisors $E_2, E_2'$ and, as before, one gets the contradiction $\Cliff_{\eta}(\O_C(E_2))=1$. It remains the possibility where $g=8$ and one between $L,L\otimes \eta$ is of type (2), while the other is of type (5) with respect to any fixed $g^1_5$.  Without loss of generality, we can assume $L=g^1_5+E_2$ and $\omega_C\otimes \eta\otimes L^\vee=g^1_5+E'_2$ for some effective degree $2$ divisors $E_2, E_2'$.
Since every base point free $g^2_7$ is simple, $L \otimes \eta$ defines a birational morphism $f \colon C \rightarrow \bar{C}\subseteq \P^2$, where $\bar{C}$ has arithmetic genus $15$ and thus its $\delta$-invariant is $\delta(\bar{C})=7$. We remark that $\bar{C}$ has at most double points as singularities because otherwise $\mathrm{gon}(C)\le4$, against our assumptions. In particular, we can fix $g^1_5= f^*(\O_{\bar{C}}(1)-n)$ where $n$ is a double point of $\bar{C}$. With this choice we have $L=f^*(\O_{\bar{C}}(1)-n)+E_2$ and $L\otimes\eta=f^*(\O_{\bar{C}}(1))$ and thus $\eta=\O_C(E_2)\otimes\nu^*(\O_{\bar{C}}(n))^{\vee}$, yielding the contradiction $\Cliff_\eta (C)\le \Cliff_{\eta}(\O_C(E_2))=1$. 
\\It remains to treat the case where $L, L \otimes \eta$ are two $g^3_9$ and $g=10$. Cases (3) and (6) of Lemma \ref{5gonal} can be ruled out because they would imply the existence of a $g^1_4$. We exclude all the other possibilities for degree reasons.

\end{proof}

\begin{remark}
    If $C$ is hyperelliptic and $\eta=\O_C(w_1+w_2+w_3-w_4-w_5-w_6)$ with $w_1,\ldots, w_6$ Weierstrass points, then $C\subset \P(H^0(\omega_C\otimes \eta)^\vee)$ has a trisecant line and thus $\Cliff_\eta(C)=2$. Indeed, the divisor $D=w_1+w_2+w_3$ satisfies $h^0(\O_C(D))=h^0(\O_C(D) \otimes \eta)=1$ and so, by Riemann-Roch, $D$ imposes one condition less than expected to $\omega_C \otimes \eta$. 
\end{remark}
\smallskip

\section{Prym-Clifford index of a general Prym curve}\label{cinque}
\subsection{Prym-Clifford index of hyperelliptic Prym curves} Let $C$ be a hyperelliptic curve, and let $M$ denote its $g^1_2$. The following result by Verra provides a description of all $2$-torsion line bundles in $\Pic^0(C)$ in terms of Weierstrass points. Let $W$ be the set of Weierstrass points of $C$, which has cardinality $2g+2$, and denote by $\mathcal P(W)$ its power set. 
\begin{lemma}\cite[Lemma 4.3]{verra}\label{sandro}
With the above notation, for every integer $1 \le k \le \frac{g+1}{2}$ let 
$$\beta_k: W_k\rightarrow \Pic^0(C)[2]$$
be the map sending a set $Z=\{w_1,\ldots,w_{2k}\}$ to the line bundle $\beta_k(Z):=M^{\otimes k}(-w_1-\cdots-w_{2k})$. Then, the following hold:
\begin{enumerate}
\item[(i)] the map $\beta_k$ is injective for $k\leq g/2$;
\item[(ii)] for $k=\frac{g+1}{2}$ the map $\beta_k$ is $2:1$;
\item[(iii)] for every non-trivial $\eta\in \Pic^0(C)[2]$ there exist $1 \le k \le \frac{g+1}{2}$ and $Z\in W_k$ such that $\eta=\beta_k(Z)$.
\end{enumerate}
\end{lemma}
We stress that the fibers of $\beta_{\frac{g+1}{2}}$ are of the form $\{Z, W\setminus Z\}$. Moreover, in part (iii) the integer $k$ is uniquely determined (and so is $Z$ if $k\leq g/2$). Indeed, suppose that $M^{\otimes k}(-p_1-\cdots-p_{2k})\simeq M^{\otimes k'}(-q_1-\cdots-q_{2k'})$ for some $k'<k\le (g+1)/2$ and  Weierstrass points $p_i, q_j$. Then we obtain an isomorphism \[
\O_C(p_1+\cdots+ p_k+q_{k'+1}+ \cdots+q_{2k'}) \simeq \O_C(q_1+\cdots+ q_{k'}+p_{k+1}+ \cdots+p_{2k})
\] and, since $k+k'\le g$, this line is special. In particular, up to subtracting possible base points, it must be a multiple of the $g^1_{2}$.
This is impossible because every divisor in $|rg^1_2|$ consists of the preimage of $r$ points of $\P^1$ and the Weierstrass points always appear with multiplicity $2$.

Thanks to the isomorphism $M^{\otimes k}(-w_1-\cdots-w_{2k})\simeq \O_C(w_1+ \cdots+ w_k-w_{k+1}- \cdots- w_{2k})$, we get the following:
\begin{corollary}
Let $C$ be a hyperelliptic curve of genus $g$ and denote by $W$ the set of its Weierstrass points. Then, every nontrivial $2$-torsion line bundle $\eta$ can be written as
\begin{equation} \label{eta}
\eta=\O_C(w_1+ \cdots+ w_k-w_{k+1}- \cdots- w_{2k})
\end{equation}
for some $1 \le k \le \lfloor \frac{g+1}{2}\rfloor$, where $\{w_1, \ldots, w_k, w_{k+1}, \ldots, w_{2k} \}\subset W$. 

If $k\le g/2$ this writing is unique up to reordering of the points $w_i$. If $k=(g+1)/2$ a subset $Z\subset W$ of cardinality $(g+1)/2$ defines the same $\eta$ as its complementary subset $W\setminus Z$.
\end{corollary}
We will now prove that the Prym-canonical Clifford index of a hyperelliptic curve $(C,\eta)$ depends on the $2$-torsion line bundle $\eta$ and, more precisely, on the integer $k$ in \eqref{eta}.

\begin{proof}[Proof of Theorem \ref{contoiper}]
Let us consider the line bundle $A=\O_C(w_1+ \ldots + w_k)$ provided by \eqref{eta}. We claim that $h^0(A)=h^0(A\otimes \eta)=1$.
Indeed, if we had $h^0(A)=r+1$ for some positive $r$, then up to subtracting its base points, $A$ would be a multiple of the $g^1_2$, which is clearly a contradiction since the $w_i$ are Weierstrass points. Hence, $\Cliff_{\eta}(A)= k -1$ and we need to show that $\Cliff_{\eta}(A)=\Cliff_{\eta}(C)$. 

Let $L$ be a line bundle computing the Prym-canonical Clifford dimension of $(C,\eta)$. Firstly, we prove that $\dim \Cliff_{\eta}(C)=(r(L),r(L\otimes \eta))=(0,0)$. By contradiction, assume that this is not the case, so that $2\le h^0(L)\le h^0(L\otimes \eta)$ by Proposition \ref{lemma1}. Since $L$ and $L\otimes \eta$ are special, base point free by Lemma \ref{lemma3} and of the same degree, then both of them should coincide with $r(L)g^1_2$ which is a contradiction as they are distinct. 

If $\Cliff_{\eta}(C)=\frac{g-1}{2}$ (which implies $g$ odd and $k=\frac{g+1}{2}$ by what we have proved above), we have done. So we may assume $\Cliff_{\eta}(C)<\frac{g-1}{2}$, or equivalently, $\deg L<\frac{g+1}{2}$; we need to prove that $L=A$. Taking $D\in |L|$ and $D'\in |L\otimes \eta|$, it is enough to show that $D$ and $D'$ are sums of Weierstrass points. Since $2D,2D'\in |L^{\otimes 2}|$ and $\deg L^{\otimes 2}<g+1$, then $L^{\otimes 2}$ is special and has no base points because $2D$ and $2D'$ have no common points according to Lemma \ref{lemma3}. Therefore, $L^{\otimes 2}$ is a multiple of the $g^1_2$ and both $2D$ and $2D'$ are sums of fibers of the morphism defined by the $g^1_2$. However, denoting by $\iota$ the hyperelliptic involution on $C$, divisors of the form $2(P + \iota(P))$ for $p \in C$ do not appear in $2D$ and $2D'$ because otherwise $L\otimes M^\vee$ (respectively, $L\otimes \eta\otimes M^\vee$)  would be effective in contradiction with $h^0(L)=h^0(L\otimes \eta)=1$. As a consequence, both $D$ and $D'$ are sums of Weierstrass points and this concludes the proof.
\end{proof}
By Lemma \ref{first lemma}, $\omega_C\otimes \eta$ is base point free if and only if $k>1$. For $k>1$, Corollary \ref{coro} yields $h^0(\omega_C\otimes \eta\otimes M^\vee)=g-3$ and we can consider the rational normal scroll containing $\varphi_{\omega_C\otimes \eta}(C)\subset \P^{g-2}$, which is defined as \[
 S= \bigcup_{D_\lambda \in |M|} \langle D_{\lambda} \rangle \subset \P^{g-2}.
 \]
 The scroll $S$ is of type $S(e_1,e_2)$ for some integers $e_1 \ge e_2 \ge 0$ satisfying $f:=e_1+e_2 =g-3$; this means that $S$ is the image of $\P(\mathcal{E})=\P(\mathcal{O}_{\P^1}(e_1)\oplus \mathcal{O}_{\P^1}(e_2))$ in $\P^{g-2}=\P^{f+1}$ through the morphism $j \colon \P(\mathcal{E})\rightarrow S \subset \P H^0(\P(\mathcal{E}), \O_{\P(\mathcal{E})}(1))=\P^{f+1}$ (we refer to \cite{schreyer} for details).

 The integers $(e_1, e_2)$ can be determined as follows (again cf. \cite{schreyer}). Set: \[ \begin{split}
d_0  & := h^0(\omega_C \otimes \eta) - h^0(\omega_C \otimes \eta \otimes M^\vee)=2, \\
d_1 & := h^0(\omega_C \otimes \eta \otimes M^\vee)- h^0(\omega_C \otimes \eta \otimes M^{-2}), \\
 & \vdots \\
 d_j & := h^0(\omega_C \otimes \eta \otimes M^{-j}) - h^0(\omega_C \otimes \eta \otimes M^{-(j+1)}), \\
 & \vdots 
 \end{split} \]
 Since $M$ has degree $2$, then $0 \le d_j \le 2$ for all $j$. For $i \in \{1,2\}$ the number $e_i$ can be computed as\[
e_i = \# \{j | d_j \ge i \}-1.
\] 
We need a preliminary lemma:
\begin{lemma} \label{lemma utile}
Let $C$ be a hyperelliptic curve and denote by $M$ its $g^1_2$.\\ A line bundle $A \in \Pic^d(C)$ with $h^0(A)=r+1$ satisfies $h^0(C, A \otimes M^\vee)= h^0(C,A)-1$ if and only if $A= M^{\otimes r}(p_1+\cdots+p_{d-2r})$, where $p_1,\ldots,p_{d-2r}$ are base points. 
\end{lemma}

\begin{proof}
The if part of the statement is obvious. \\For the converse implication, up to replacing $|A|$ with its base-point free linear subsystem, we can assume that $|A|$ is base-point free. Since $h^0(C, A \otimes M^\vee)= h^0(C,A)-1$ by hypothesis, then the morphism $\varphi_A$ defined by $A$ does not separate pairs of points which are conjugate under the hyperelliptic involution, that is, $\varphi_A$ factors through the morphism $f \colon C \rightarrow \P^1$ defined by the $g^1_2$. This is equivalent to $A=M^{\otimes r}$.
\end{proof}

\begin{proposition} \label{prop scroll}
Let $2 \le k \le \lfloor \frac{g+1}{2} \rfloor$ and $\eta= \O_C(w_1+ \ldots +w_k-w_{k+1}- \ldots -w_{2k}) \in \Pic^0(C)$ be a non trivial $2$-torsion line bundle. Then, the scroll $S\supset \varphi_{\omega_C\otimes \eta}(C)$ is of type $S=S(g-1-k, k-2)$.
\end{proposition}
\begin{proof}
With the notation introduced above, let $i$ be the first integer such that $d_i=1$. Since $h^0(\omega_C \otimes \eta\otimes M^{-i})=g-1-2i$ and $h^0(\omega_C \otimes \eta\otimes M^{-(i+1)})=g-2-2i$, Lemma \ref{lemma utile} yields 
\begin{equation}\label{joker}\omega_C \otimes \eta\otimes M^{-i}=M^{\otimes (g-2-2i)}+p_1+\ldots+p_{2i+2}\end{equation} where $p_1, \ldots, p_{2i+2}$ are base points. We get that $d_l=1$ for all $ i \le l \le g-2-i$ so that $S$ is of type $S(g-2-i, i-1)$. It only remains to show that $i=k-1$. Since $\omega_C = (g-1)g^1_2$, from \eqref{joker} we get \[
M^{\otimes (g-1-i)} \otimes \eta= M^{\otimes (g-2-2i)}+p_1 + \ldots + p_{2i+2},
\] and thus \[
\eta= M^{\otimes (i+1)} -(p_1 + \ldots + p_{2i+2}).
\] Since $\eta^{\otimes 2} \simeq \O_C$, it follows that $p_1, \ldots ,p_{2i+2}$ are Weierstrass points and thus \[
\eta=\O_C(p_1+ \ldots+ p_{i+1}-p_{i+2}- \ldots - p_{2i}),
\] which concludes the proof.
\end{proof}

We will now apply Park's results \cite{parkhyper}, which determine the syzygies of an embedded hyperelliptic curve $C\subset \P^r$ only from the type of scroll containing it. 

We recall that, given a very ample line bundle $L$ on a hyperelliptic curve $C$, the embedded curve $C\subset \P^r=\P(H^0(C,L)^{\vee})$ is said to be \emph{$m$-regular in the sense of Castelnuovo--Mumford} if $m$ is the smallest integer such that the ideal sheaf $I_{C/\P^r}$ is $m$-regular, that is, $H^1(\P,I_{C/\P^r}(m-1))=0$. Equivalently, considering a  minimal graded free resolution of $R(C):= \bigoplus_n H^0(C, L^{\otimes n})$ as a finitely generated module over $R:=\oplus_n\mathrm{Sym}^n(H^0(C,L))$ :\[
\cdots \to \oplus _jR(-i-j)^{\beta_{i,j}}\to \cdots \to \oplus _jR(-1-j)^{\beta_{1,j}} \to \oplus _jR(-j)^{\beta_{0,j}}\to R(C),
\]
the curve $C\subset \P^r$ is $m$-regular if $\beta_{i,j}=0$ for all $i\ge 0$ and $j\ge m$.

Analogously, having fixed integers $k \ge 2$ and $p \ge 1$, property $(N_{k,p})$ for $C\subset \P^r$ is defined by the vanishing $\beta_{i,j}(C)=0$ for $1 \le i \le p$ and all $j \ge k$. 

Applying \cite[Thm 1.3]{parkhyper} and Proposition \ref{prop scroll}, we get the following result: \begin{corollary}
    Let $C$ be a hyperelliptic curve of genus $g$ and let $\eta= \O_C(w_1+ \ldots+w_k-w_{k+1}- \ldots-w_{2k})$ with $w_1, \ldots, w_{2k}$ Weierstrass points and $3 \le k \le \lfloor \frac{g+1}{2} \rfloor$. Set \[\nu:= \begin{cases}
         5 \quad \mathrm{for} \ k=3 \\
         4 \quad \mathrm{for} \ k=4 \\
         3 \quad \mathrm{for} \ k\ge 5 \\
     \end{cases} \quad \mathrm{and} \quad p:= \nu (k-2)-2k+1.\] Then: \begin{enumerate}
     \item[(i)] the Castelnuovo--Mumford regularity of $\varphi_{\omega_C\otimes\eta}(C)$ is $\nu+1$;
     \item[(ii)] $\varphi_{\omega_C\otimes\eta}(C)$ satisfies $N_{\nu,p}$ but fails property $N_{\nu, p+1}$.
     \end{enumerate}
\end{corollary}
\begin{proof}
    We assume that $\omega_C \otimes \eta$ is very ample, or equivalently, that $k \ge3$. By Proposition \ref{prop scroll}, the Prym-canonical curve $\varphi_{\omega_C\otimes\eta}(C)$ is contained in the rational normal scroll $S=(g-1-k, k-2)$. Following Park's terminology, this is equivalent to saying that the factorization type of $\omega_C \otimes \eta$ is $(m,b)=(g-k-1,2k)$, that is, $m=g-k-1$ is the biggest integer $t$ such that $\omega_C\otimes \eta \otimes M^{-t}$ is effective.Park associates to the factorization type $(m,b)$ two integers $\nu$ and $p$, which govern the syzygies of the embedded curve. In our situation a direct computation yields: 
\[ \begin{split}
\nu := \Big\lceil \frac{b-1}{m+b-g-1} \Big \rceil& = \begin{cases}
    5 \quad \mathrm{for} \quad k=3\\
    4 \quad \mathrm{for} \quad k=4\\
    3 \quad \mathrm{for} \quad k\ge5\\
\end{cases}
\\
p:=\nu(m+b-g-1)-b+1& = \nu(k-2)-2k+1.
\end{split}
\] 
The result then follows from \cite[Thm.1.3]{parkhyper}.
\end{proof}

\begin{remark}
    Applying \cite[Lem.4.2]{parkhyper}, all graded Betti numbers of $C\subset \P^{g-2}$ can be explicitly computed in terms of $m$ and $b$. 

\end{remark}

\subsection{The case of general curves}
Let $\mathcal{R}_g$ be the moduli space of Prym curves and consider a general element $(C, \eta) \in \mathcal{R}_g$. In this section we compute $\Cliff_{\eta}(C)$. Firstly, we show that the function $
 \Cliff_{\eta}(C)
$ is lower semicontinuous in families of Prym curves.  Then, the computation of the Prym-canonical Clifford index for a general curve follows directly from the case of hyperelliptic curves.

Given a family $p: \mathcal{C}\rightarrow B$ of smooth curves of genus $g>1$ over a $1$-dimensional base $B$, we consider the \textit{relative symmetric product} \[
\mathcal{C}_d:=\mathrm{Hilb}^d_{\mathcal{C}/B}\longrightarrow B, \] and the \textit{relative Picard variety} \[
\textbf{Pic}^0(p)\stackrel{q}{\longrightarrow} B.
\]
By definition, to give a family of Prym curves, one should also assign a line bundle $\eta$ over $\mathcal C$ defining a section of $q$ and an isomorphism $\beta: \eta^{\otimes 2}\to \mathcal O_{\mathcal C}$.

\begin{proposition} \label{lowcont}
The function \[
(C, \eta) \rightarrow \Cliff_{\eta}(C)
\] is lower semicontinuous in families of Prym curves.
\end{proposition}

\begin{proof}
We denote by $(C_b, \eta_b)$ a general fiber of a smooth family of Prym curves $(p \colon \mathcal{C} \rightarrow B,\eta,\beta)$ as above, and by $(C_0,\eta_0)$ the special fiber. Let $L_b\in \Pic(C_b)$ be a line bundle on $C_b$ such that $\Cliff_{\eta_b}(C_b)=\Cliff_{\eta_b}(L_b)$. Write $L_b=\O_{C_b}(D_b)$ for some effective divisor $D_b$ of degree $d$ on $C_b$, so that  $\eta_b= \O_{C_b}(D_b-E_b)$ for some other degree $d$ effective divisor $E_b$ on $C_b$; in particular, $\eta_b$ lies in the image of the difference map $\phi_b \colon (C_{b})_d \times (C_{b})_d \rightarrow \Pic^0(C_b)$. We have the following diagram: \begin{equation} \label{diagram}
\begin{tikzcd}
\mathcal{C}_d \times \mathcal{C}_d \arrow[r, "\phi"] \arrow[rd, "\pi_d"] & \mathbf{Pic}^0(p) \arrow[d, "q"] \\
                                                                         & B                                 
\end{tikzcd}
\end{equation} where $\phi$ is the relative difference map and $\pi_d$ is induced by $p$. By the smoothness of the family $p$, the divisors $D_b, E_b$ can be extended to relative effective divisors $D,E\subset \mathcal C$ over $B$ such that $\eta=\mathcal O_{\mathcal C}(D-E)$. Over the special point $0 \in B$ we thus get $\eta_0= \O_{C_0}(D_0-E_0) \in \Pic^0(C_0)$, that is, $\eta_0 \in \mathrm{Im}{\phi_0}$. Setting $L_0:=\O_{C_0}(D_0)$ and $L_0 \otimes \eta_0:=\O_{C_0}(E_0)$, we obtain \[
\begin{split}
\Cliff_{\eta_0}(L_0) & = \deg(L_0)-h^0(L_0)-h^0(L_0 \otimes \eta_0)+1 \le \\
& \le \deg(L_b)-h^0(L_b)-h^0(L_b \otimes \eta_b)+1= \Cliff_{\eta_b}(L_b)
\end{split}
\] because $h^0(L_0)\ge h^0(L_b)$ and $h^0(L_0 \otimes \eta_0) \ge h^0(L_b \otimes \eta_b)$ by upper semicontinuity. This concludes the proof.
\end{proof}

As a corollary, we provide the following:
\begin{proof}[Proof of Theorem \ref{Cgen}]
Let $(p \colon \mathcal{C} \rightarrow B,\eta,\beta)$ be a smooth family of Prym curves such that the central fiber $C_0=p^{-1}(0)$ is hyperelliptic and the line bundle $\eta_0$ can be written as in \eqref{eta} with $k=\lfloor \frac{g+1}{2} \rfloor$. Then, Theorem \ref{contoiper} and Proposition \ref{lowcont} imply that \[
\lfloor \frac{g-1}{2} \rfloor = \Cliff_{\eta_0}(C_0) \le \Cliff_{\eta_b}(C_b).
\] Conversely, Remark \ref{surj_diff_map} yields \[
\Cliff_{\eta_b}(C_b) \le \lfloor \frac{g-1}{2} \rfloor,
\] and thus equality holds. Again by Remark \ref{surj_diff_map} we conclude that $\dim\Cliff_{\eta_b}(C_b)=(0,0)$
\end{proof}

\bibliographystyle{alpha}
\bibliography{Biblio}

@book {acgh,
    AUTHOR = {Arbarello, E. and Cornalba, M. and Griffiths, P. A. and
              Harris, J.},
     TITLE = {Geometry of algebraic curves. {V}ol. {I}},
    SERIES = {Grundlehren der mathematischen Wissenschaften [Fundamental
              Principles of Mathematical Sciences]},
    VOLUME = {267},
 PUBLISHER = {Springer-Verlag, New York},
      YEAR = {1985},
     PAGES = {xvi+386},
      ISBN = {0-387-90997-4},
   MRCLASS = {14Hxx (14-02)},
  MRNUMBER = {770932},
MRREVIEWER = {Werner\ Kleinert},
       DOI = {10.1007/978-1-4757-5323-3},
       URL = {https://doi.org/10.1007/978-1-4757-5323-3},
}

@article {cdgk,
    AUTHOR = {Ciliberto, Ciro and Dedieu, Thomas and Galati, Concettina and
              Knutsen, Andreas Leopold},
     TITLE = {On the locus of {P}rym curves where the {P}rym-canonical map
              is not an embedding},
   JOURNAL = {Ark. Mat.},
  FJOURNAL = {Arkiv f\"or Matematik},
    VOLUME = {58},
      YEAR = {2020},
    NUMBER = {1},
     PAGES = {71--85},
      ISSN = {0004-2080,1871-2487},
   MRCLASS = {14H40 (14H10)},
  MRNUMBER = {4094639},
MRREVIEWER = {Pawe\l\ Bor\'owka},
       DOI = {10.4310/arkiv.2020.v58.n1.a5},
       URL = {https://doi.org/10.4310/arkiv.2020.v58.n1.a5},
}

@book{hartshorne,
  author    = {Robin Hartshorne},
  title     = {Algebraic Geometry},
  series    = {Graduate Texts in Mathematics},
  volume    = {52},
  publisher = {Springer-Verlag},
  address   = {New York, NY},
  year      = {1977},
  isbn      = {978-0-387-90244-9},
  doi       = {10.1007/978-1-4757-3849-0}
}

@incollection {verra,
    AUTHOR = {Verra, Alessandro},
     TITLE = {Rational parametrizations of moduli spaces of curves},
 BOOKTITLE = {Handbook of moduli. {V}ol. {III}},
    SERIES = {Adv. Lect. Math. (ALM)},
    VOLUME = {26},
     PAGES = {431--506},
 PUBLISHER = {Int. Press, Somerville, MA},
      YEAR = {2013},
      ISBN = {978-1-57146-259-6},
   MRCLASS = {14H10 (14E08 14M20)},
  MRNUMBER = {3135442},
MRREVIEWER = {Gianfranco\ Casnati},
}

@article {elms,
    AUTHOR = {Eisenbud, David and Lange, Herbert and Martens, Gerriet and
              Schreyer, Frank-Olaf},
     TITLE = {The {C}lifford dimension of a projective curve},
   JOURNAL = {Compositio Math.},
  FJOURNAL = {Compositio Mathematica},
    VOLUME = {72},
      YEAR = {1989},
    NUMBER = {2},
     PAGES = {173--204},
      ISSN = {0010-437X,1570-5846},
   MRCLASS = {14H45 (14C20 14J28)},
  MRNUMBER = {1030141},
MRREVIEWER = {Olivier\ Debarre},
       URL = {http://www.numdam.org/item?id=CM_1989__72_2_173_0},
}

@article {cm,
    AUTHOR = {Coppens, Marc and Martens, Gerriet},
     TITLE = {Secant spaces and {C}lifford's theorem},
   JOURNAL = {Compositio Math.},
  FJOURNAL = {Compositio Mathematica},
    VOLUME = {78},
      YEAR = {1991},
    NUMBER = {2},
     PAGES = {193--212},
      ISSN = {0010-437X,1570-5846},
   MRCLASS = {14H10 (14C20 14H40 14H45)},
  MRNUMBER = {1104787},
MRREVIEWER = {Montserrat\ Teixidor i Bigas},
       URL = {http://www.numdam.org/item?id=CM_1991__78_2_193_0},
}

@article {lm,
    AUTHOR = {Lange, H. and Martens, G.},
     TITLE = {On the gonality sequence of an algebraic curve},
   JOURNAL = {Manuscripta Math.},
  FJOURNAL = {Manuscripta Mathematica},
    VOLUME = {137},
      YEAR = {2012},
    NUMBER = {3-4},
     PAGES = {457--473},
      ISSN = {0025-2611,1432-1785},
   MRCLASS = {14H51},
  MRNUMBER = {2875287},
MRREVIEWER = {Jos\'e\ Ignacio\ Farr\'an},
       DOI = {10.1007/s00229-011-0475-4},
       URL = {https://doi.org/10.1007/s00229-011-0475-4},
}

@article {ms,
    AUTHOR = {Martens, G. and Schreyer, F.-O.},
     TITLE = {Line bundles and syzygies of trigonal curves},
   JOURNAL = {Abh. Math. Sem. Univ. Hamburg},
  FJOURNAL = {Abhandlungen aus dem Mathematischen Seminar der Universit\"at
              Hamburg},
    VOLUME = {56},
      YEAR = {1986},
     PAGES = {169--189},
      ISSN = {0025-5858,1865-8784},
   MRCLASS = {14H45 (14H50)},
  MRNUMBER = {882414},
MRREVIEWER = {Rafael\ Hernandez},
       DOI = {10.1007/BF02941515},
       URL = {https://doi.org/10.1007/BF02941515},
}

@article {cm00,
    AUTHOR = {Coppens, Marc and Martens, Gerriet},
     TITLE = {Linear series on 4-gonal curves},
   JOURNAL = {Math. Nachr.},
  FJOURNAL = {Mathematische Nachrichten},
    VOLUME = {213},
      YEAR = {2000},
     PAGES = {35--55},
      ISSN = {0025-584X,1522-2616},
   MRCLASS = {14H51 (14H45)},
  MRNUMBER = {1755245},
MRREVIEWER = {James\ N.\ Brawner},
       DOI =
              {10.1002/(SICI)1522-2616(200005)213:1<35::AID-MANA35>3.3.CO;2-Q},
       URL =
              {https://doi.org/10.1002/(SICI)1522-2616(200005)213:1<35::AID-MANA35>3.3.CO;2-Q},
}

@article {park,
    AUTHOR = {Park, S.-S.},
     TITLE = {On the variety of special linear series on a general 5-gonal
              curve},
   JOURNAL = {Abh. Math. Sem. Univ. Hamburg},
  FJOURNAL = {Abhandlungen aus dem Mathematischen Seminar der Universit\"at
              Hamburg},
    VOLUME = {72},
      YEAR = {2002},
     PAGES = {283--291},
      ISSN = {0025-5858,1865-8784},
   MRCLASS = {14H51},
  MRNUMBER = {1941560},
MRREVIEWER = {Marina\ Zompatori},
       DOI = {10.1007/BF02941678},
       URL = {https://doi.org/10.1007/BF02941678},
}

@article {ckm,
    AUTHOR = {Coppens, Marc and Keem, Changho and Martens, Gerriet},
     TITLE = {Primitive linear series on curves},
   JOURNAL = {Manuscripta Math.},
  FJOURNAL = {Manuscripta Mathematica},
    VOLUME = {77},
      YEAR = {1992},
    NUMBER = {2-3},
     PAGES = {237--264},
      ISSN = {0025-2611,1432-1785},
   MRCLASS = {14H40 (14C20 14H45)},
  MRNUMBER = {1188583},
MRREVIEWER = {Montserrat\ Teixidor i Bigas},
       DOI = {10.1007/BF02567056},
       URL = {https://doi.org/10.1007/BF02567056},
}

@article{bdc,
  AUTHOR  = {Fabio Bardelli and Andrea Del Centina},
  TITLE   = {Osservazioni sullo spazio dei moduli delle curve trigonali},
  JOURNAL = {Atti della Accademia Nazionale dei Lincei. Classe di Scienze Fisiche, Matematiche e Naturali. Rendiconti, Serie 8},
  VOLUME  = {70},
  NUMBER  = {2},
  PAGES   = {96--100},
  YEAR    = {1981}
}

@article {parkhyper,
    AUTHOR = {Park, Euisung},
     TITLE = {Higher syzygies of hyperelliptic curves},
   JOURNAL = {J. Pure Appl. Algebra},
  FJOURNAL = {Journal of Pure and Applied Algebra},
    VOLUME = {214},
      YEAR = {2010},
    NUMBER = {2},
     PAGES = {101--111},
      ISSN = {0022-4049,1873-1376},
   MRCLASS = {14H51 (13D02)},
  MRNUMBER = {2559684},
MRREVIEWER = {Peter\ John\ Vermeire},
       DOI = {10.1016/j.jpaa.2009.04.006},
       URL = {https://doi.org/10.1016/j.jpaa.2009.04.006},
}

@article {greenlazarsfeld,
    AUTHOR = {Green, Mark and Lazarsfeld, Robert},
     TITLE = {On the projective normality of complete linear series on an
              algebraic curve},
   JOURNAL = {Invent. Math.},
  FJOURNAL = {Inventiones Mathematicae},
    VOLUME = {83},
      YEAR = {1986},
    NUMBER = {1},
     PAGES = {73--90},
      ISSN = {0020-9910,1432-1297},
   MRCLASS = {14H10 (14H45)},
  MRNUMBER = {813583},
MRREVIEWER = {H.\ Lange},
       DOI = {10.1007/BF01388754},
       URL = {https://doi.org/10.1007/BF01388754},
}

@article{schreyer,
AUTHOR = {Schreyer, F.-O.},
JOURNAL = {Mathematische Annalen},
PAGES = {105-138},
TITLE = {Syzygies of Canonical Curves and Special Linear Series.},
URL = {http://eudml.org/doc/164137},
VOLUME = {275},
YEAR = {1986},
}

@article{SP01,
  AUTHOR  = {Park, Seong-Suk},
  TITLE   = {Existence of base-point-free pencils of degree {$g-1$} on bi-elliptic curves},
  JOURNAL = {Osaka Journal of Mathematics},
  VOLUME  = {40},
  NUMBER  = {1},
  YEAR    = {2003},
  PAGES   = {279--285}
}

@misc{farkaslc,
      TITLE={Secant loci on moduli of Prym varieties}, 
      AUTHOR={Gavril Farkas and Margherita Lelli-Chiesa},
      YEAR={2025},
      EPRINT={2509.26118},
      ARCHIVEPREFIX={arXiv},
      PRIMARYCLASS={math.AG},
      URL={https://arxiv.org/abs/2509.26118}, 
}

@article {green,
    AUTHOR = {Green, Mark L.},
     TITLE = {Koszul cohomology and the geometry of projective varieties},
      NOTE = {With an appendix by Robert Lazarsfeld and Green},
   JOURNAL = {J. Differential Geom.},
  FJOURNAL = {Journal of Differential Geometry},
    VOLUME = {19},
      YEAR = {1984},
    NUMBER = {1},
     PAGES = {125--171},
      ISSN = {0022-040X,1945-743X},
   MRCLASS = {14F05 (14B12)},
  MRNUMBER = {739785},
MRREVIEWER = {G.\ Horrocks},
       URL = {http://projecteuclid.org/euclid.jdg/1214438426},
}

@article {langesernesi,
    AUTHOR = {Lange, H. and Sernesi, E.},
     TITLE = {Quadrics containing a {P}rym-canonical curve},
   JOURNAL = {J. Algebraic Geom.},
  FJOURNAL = {Journal of Algebraic Geometry},
    VOLUME = {5},
      YEAR = {1996},
    NUMBER = {2},
     PAGES = {387--399},
      ISSN = {1056-3911,1534-7486},
   MRCLASS = {14H40 (14C34 14K05)},
  MRNUMBER = {1374713},
MRREVIEWER = {Montserrat\ Teixidor i Bigas},
}

@article{voisinodd,
  AUTHOR  = {Voisin, Claire},
  TITLE   = {Green's canonical syzygy conjecture for generic curves of odd genus},
  JOURNAL = {Compositio Mathematica},
  VOLUME  = {141},
  NUMBER  = {5},
  PAGES   = {1163--1190},
  YEAR    = {2005},
  DOI     = {10.1112/S0010437X05001387},
}

@article{voisineven,
  AUTHOR  = {Voisin, Claire},
  TITLE   = {Green's generic syzygy conjecture for curves of even genus lying on a $K3$ surface},
  JOURNAL = {Journal of the European Mathematical Society},
  VOLUME  = {4},
  NUMBER  = {4},
  PAGES   = {363--404},
  YEAR    = {2002},
  DOI     = {10.1007/s100970200042},
}

@article {papadima,
    AUTHOR = {Aprodu, Marian and Farkas, Gavril and Papadima, \c Stefan and
              Raicu, Claudiu and Weyman, Jerzy},
     TITLE = {Koszul modules and {G}reen's conjecture},
   JOURNAL = {Invent. Math.},
  FJOURNAL = {Inventiones Mathematicae},
    VOLUME = {218},
      YEAR = {2019},
    NUMBER = {3},
     PAGES = {657--720},
      ISSN = {0020-9910,1432-1297},
   MRCLASS = {14H51 (13D02 14M15 14N05 20G05)},
  MRNUMBER = {4022070},
MRREVIEWER = {Yeongrak\ Kim},
       DOI = {10.1007/s00222-019-00894-1},
       URL = {https://doi.org/10.1007/s00222-019-00894-1},
}

@article{kemeny,
  AUTHOR  = {Michael Kemeny},
  TITLE   = {A proof of generic Green's conjecture in odd genus },
  JOURNAL = {Épijournal de Géométrie Algébrique (EPIGA)},
  VOLUME = {9},
  YEAR    = {2025},
  NUMBER  = {27}
}

\end{document}